\documentclass{ThesisMathTemplate}
\title{On parameterizations of cyclic $N$-isogenies and strict $K$-curves lying above rational points of $Y_0^+(N)$}
\hypersetup{
    colorlinks=true,
    linkcolor=blue,
    filecolor=magenta,      
    urlcolor=blue,
}
\usepackage{adjustbox}
\usepackage{gensymb}
\usepackage{array}
\usepackage{makecell}

\usepackage[backend=biber, style=alphabetic]{biblatex}
\nc{\vol}{\on{vol}}
\nc{\Sim}{\on{Sim}}
\nc{\rank}{\on{rank}}
\nc{\C}{\mathbb{C}}
\nc{\Z}{\mathbb{Z}}
\nc{\Q}{\mathbb{Q}}

\nc{\SL}{\on{SL}}
\usepackage[OT2,T1]{fontenc}
\DeclareSymbolFont{cyrletters}{OT2}{wncyr}{m}{n}
\DeclareMathSymbol{\Sha}{\mathalpha}{cyrletters}{"58}
\nc{\Tor}{\on{Tor}}
\begin{document}
	\maketitle

\tableofcontents

\pagebreak

\section*{Acknowledgments}
\thispagestyle{empty}

I thank my thesis advisor Professor Noam Elkies for our many conversations and correspondences, for his gracious willingness to give advice and feedback, and for introducing me to the topics of this thesis. I also thank Benjamin Peirce Fellow Fabian Gundlach for his work with me during Summer 2020; without the background on elliptic curves I gained during this time, this thesis would not have been possible. I thank Alec Sun for help with proofreading. Finally, I thank my parents and sister for their continued support of my mathematical pursuits. 

\clearpage

\section{Introduction}

\subsection{Cyclic isogenies}

Modular curves are moduli spaces for elliptic curves: a point of a modular curve corresponds to an elliptic curve enhanced with certain torsion data. For example, the curve $X(1)$ corresponds to $\C$-isomorphism classes of curves, and the curves $X_0(N)$ correspond to such classes paired with additional data of a cyclic $N$-isogeny. Therefore, understanding the modular curves $X_0(N)$ is useful for understanding cyclic isogenies, which have many applications in elliptic curve computations. One of the most celebrated applications is the theory of descent, which often allows the computation of rank, conductor, and other invariants of interest for elliptic curves \cite{MR2514094}. The Schoof-Elkies-Atkin algorithm, which refines Schoof's algorithm for point-counting on elliptic curves over finite fields, relies on explicitly identifying cyclic isogenies and their kernels \cite{MR1486831}. Studying cyclic isogenies is a useful tool for finding elliptic curves with complex multiplication (CM), since any curve with a nontrivial cyclic self-isogeny is necessarily CM. 

Since $X(1)$ is a rational curve, the correspondence between elliptic curves and $X(1)$ can be reinterpreted as a rational parametrization of elliptic curves---that is, we associate to every point of $\C$ a $\C$-isomorphism class of elliptic curves. This correspondence is given explicitly by the $j$-invariant. When $X_0(N)$ has genus $0$, the same principle may be applied, and we obtain a correspondence between $\C$ and isomorphism classes of cyclic $N$-isogenies $E \to E'$. Analogously to the $j$-invariant, this correspondence may be explicitly given by a \emph{Hauptmodul} $h_N$ of $X_0(N)$. This Hauptmodul may be used to translate formulas involving $q$-expansions into rational functions involving $h_N$, clarifying data and providing a useful computational tool. 

Elliptic $K$-curves are closely related to the subject of cyclic isogenies. These are curves $E$ defined over some extension of a field $K$ that are isogenous to all of their Galois conjugates $E^\sigma$. They arise as a natural generalization of curves defined over $K$. In particular, $\Q$-curves have been extensively studied. Which special properties of elliptic curves over $\Q$, such as modularity, extend to $\Q$-curves \cite{ellenberg}? The connection to cyclic isogenies comes from a theorem of Elkies, which states that all non-CM $K$-curves arise from $K$-rational points on the quotient $X^*(N)$ of $X_0(N)$ by the action of the Atkin-Lehner involutions. 

\subsection{This thesis}

This thesis presents two original results. First, we tabulate explicit formulas for the coefficients of cyclic $N$-isogenous elliptic curves
\begin{align*}
	E: y^2 = x^3 - \frac{A_4}{48}x + \frac{A_6}{864} && E': y^2 = x^3 - \frac{A_4'}{48}x + \frac{A_6'}{864}
\end{align*}
in terms of a distinguished Hauptmodul on all modular curves $X_0(N)$ of genus $0$. To this end, we provide an exposition on the derivation of these Hauptmoduln as products of the Dedekind eta function based on the approach of \cite{MR0422158}. We also include an abbreviated table of rational functions for the $j$-invariant in terms of Hauptmoduln, and we discuss an application of these expressions to finding special values of the $j$-invariant at CM points.

The second original result of this thesis is a new theorem on $K$-curves given by a $K$-rational orbit $\{\tau, -1/N\tau\}$ of the Fricke involution on $Y_0(N)$. Points on $Y_0(N)$ correspond to isomorphism classes of isogenies over $\C$, but these classes split into many twists when considered over a number field $L$. We prove the following theorem:
 \begin{theorem*}
	 Let $L/K$ be a quadratic extension, and let $\{h, h'\} \subset Y_0(N)(L) \setminus Y_0(N)(K)$ be the fiber of a non-CM point in $Y_0^+(N)(K)$. Let $E, E'$ be the corresponding isogenous curves. Then the following are equivalent:
	 \begin{enumerate}[(i)]
	 \item 
There exists a choice of twist of the cyclic $N$-isogeny $E \to E'$ defined over $L$ such that $E'$ is isomorphic over $L$ to the conjugate curve $\conj{E}$ of $E$.
\item At least one of the integers $N$ or $-N$ lies in the image of the Galois norm map $L^\times \to K^\times$. 
	 \end{enumerate}
 \end{theorem*} 
Moreover, we give an explicit construction for all twists satisfying condition (i). In general, the isogenies between an elliptic $K$-curve and its conjugates might only be defined over some separable extension of $L = \Q(E)$. We lead into this theorem with background on elliptic $K$-curves.

Throughout, we assume basic familiarity with the theory of elliptic curves and modular forms. However, we provide background on the role of modular curves as moduli spaces, including the fundamental correspondence, the Atkin-Lehner involutions, the treatment of modular curves over number fields, and the Tate curve. 
\section{Modular curves}

\subsection{Modular curves as moduli spaces for complex elliptic curves}

Let $\mathbb{H} \subset \C$ denote the complex upper half-plane and let $\SL_2(\Z)$ act on $\mathbb{H}$ by fractional linear transformations. A modular curve $Y$ is a Riemann surface given by a quotient space of $\mathbb{H}$ by the action of certain subgroups $\Gamma \subseteq \SL_2(\Z)$ known as congruence subgroups. 

\defn{Let $N$ be a positive integer. We define three classes of congruence subgroups:

\begin{align*}
	\Gamma(N) &= \left\{\begin{pmatrix}
		a & b \\ c & d 
	\end{pmatrix} \in \SL_2(\Z) : a \equiv d \equiv 1 \bmod N, c \equiv b \equiv 0 \bmod N\right\}\\
	\Gamma_1(N) &= \left\{\begin{pmatrix}
		a & b \\ c & d 
	\end{pmatrix} \in \SL_2(\Z) : a \equiv d \equiv 1 \bmod N, c \equiv  0 \bmod N\right\}\\
	\Gamma_0(N) &= \left\{\begin{pmatrix}
		a & b \\ c & d 
	\end{pmatrix} \in \SL_2(\Z) : c \equiv  0 \bmod N\right\}
\end{align*}
More generally, a congruence subgroup is defined as a subgroup $\Gamma \subseteq \SL_2(\Z)$ containing $\Gamma(N)$ for some positive integer $N$.}

\defn{We denote the modular curves corresponding to the three congruence groups above by
\begin{align*}
	Y(N) &= \BH/\Gamma(N),\\
	Y_1(N) &= \BH/\Gamma_1(N),\\
	Y_0(N) &= \BH/\Gamma_0(N).
\end{align*}}

It is also useful to consider the compactifications of these curves. Let $\BH^* = \BH \cup \BQ \cup \{\infty\}$ denote the extended upper half-plane. Since $\SL_2(\Z)$ also acts on $\BP^1_\BQ = \BQ \cup \{\infty\}$ by fractional linear transformations, its action on $\BH$ extends to $\BH^*$. 

\defn{We define
\begin{align*}
 	X(N) &= \BH^*/\Gamma(N),\\
 	X_1(N) &= \BH^*/\Gamma_1(N),\\
 	X_0(N) &= \BH^*/\Gamma_0(N),
 \end{align*}
and we also call these modular curves. These are compact Riemann surfaces, and orbit classes of $\BP^1_\BQ$ are called \emph{cusps}.}

The inclusions $\Gamma(N) \subseteq \Gamma_1(N) \subseteq\Gamma_0(N)$ induce natural maps $Y(N) \to Y_1(N) \to Y_0(N)$, and likewise for their compactified versions. Furthermore, if $N' \mid N$, then the inclusion $\Gamma(N) \subseteq \Gamma(N')$ induces a map $Y(N) \to Y(N')$; likewise for $\Gamma_0$ and $\Gamma_1$. 

These classes of modular curves can be interpreted as moduli spaces for \emph{enhanced elliptic curves}---that is, each point of a modular curve corresponds to an isomorphism class of elliptic curves $\C/\Lambda$ together with a certain type of torsion data, and this correspondence is natural in the sense that it commutes with the maps between the modular curves. 

For $\tau \in \BH$, let $\Lambda_\tau$ denote the lattice $\Z + \tau \Z$. We define three classes of enhanced elliptic curves corresponding to the above three classes of congruence subgroups.

\defn{An \emph{enhanced elliptic curve for $\Gamma_0(N)$} is a pair $(E, C)$, where $E$ is a complex elliptic curve and $C$ is a cyclic subgroup of $E$ of order $N$. We define $(E, C) \sim (E', C')$ if and only if there exists some isomorphism over $\C$ taking $E$ to $E'$ and $C$ to $C'$. Then we define
\begin{align*}
	S_0(N) = \{(E, C) \text{ enhanced elliptic curves for } \Gamma_0(N)\}/\sim.
\end{align*}

\defn{The \emph{degree} of an isogeny is the order of its kernel; if an isogeny has degree $N$, we call it an \emph{$N$-isogeny}. We say that an isogeny is \emph{cyclic} if it has cyclic kernel.} 

Instead of using the term ``enhanced elliptic curve for $\Gamma_0(N)$'' we will usually refer to $(E, C)$ by its corresponding cyclic $N$-isogeny given by the quotient $E \to E/C$.

\defn{An \emph{enhanced elliptic curve for $\Gamma_1(N)$} is a pair $(E, Q)$, where $E$ is a complex elliptic curve and $Q\in E$ is a point of order $N$. We define $(E, Q) \sim (E', Q')$ if and only if there exists some isomorphism taking $E$ to $E'$ and $Q$ to $Q'$. Then we define
\begin{align*}
	S_1(N) = \{(E, Q) \text{ enhanced elliptic curves for } \Gamma_1(N)\}/\sim.
\end{align*}}
\defn{An \emph{enhanced elliptic curve for $\Gamma(N)$} is a pair $(E, (P,Q))$, where $E$ is a complex elliptic curve and $P, Q$ are a basis for the $N$-torsion subgroup $E[N]$ of $E$ with Weil pairing $e_N(P, Q) = e^{2\pi i/N}$.  We define $(E, (P, Q)) \sim (E', (P', Q'))$ if and only if there exists some isomorphism taking $E$ to $E'$ and both $P$ to $P'$ and $Q$ to $Q'$. Then we define
\begin{align*}
	S(N) = \{(E, (P, Q)) \text{ enhanced elliptic curves for } \Gamma(N)\}/\sim.
\end{align*}}

\begin{theorem}\label{modular-correspondence}

	\begin{enumerate}[(a)]
		\item The points of $Y_0(N)$ correspond bijectively to elements of $S_0(N)$ under the map
		\begin{align*}
			\tau \leftrightarrow (\C/\Lambda_\tau, \wangle{1/N}).
		\end{align*}
		\item The points of $Y_1(N)$ correspond bijectively to elements of $S_1(N)$ under the map
		\begin{align*}
			\tau \leftrightarrow (\C/\Lambda_\tau, 1/N).
		\end{align*}
		\item The points of $Y_0(N)$ correspond bijectively to elements of $S_0(N)$ under the map
		\begin{align*}
			\tau \leftrightarrow (\C/\Lambda_\tau, (1/N, \tau/N)).
		\end{align*}
	\end{enumerate}
\end{theorem}
\begin{proof}
	See [\cite{MR2112196}, Theorem 1.5.1].
\end{proof}
\begin{corollary}
	The points of $\Gamma(1) = \Gamma_1(1) = \Gamma_0(1)$ correspond bijectively to $\C$-isomorphism classes of elliptic curves (without additional torsion data) via the map
	\begin{align*}
		\tau \leftrightarrow \C/\Lambda_\tau.
	\end{align*}
\end{corollary}

By part (a) of Theorem \ref{modular-correspondence}, we may interpret $Y_0(N)$ as parameterizing isomorphism classes of cyclic $N$-isogenies over $\C$. 

\subsection{The Fricke and Atkin-Lehner involutions}\label{fricke}

Any isogeny $\varphi: E \to E'$ of degree $N$ has a unique dual isogeny $\hat{\varphi}: E \to E'$ such that \begin{align*}\hat{\varphi} \circ \varphi &= [N]_{E}\\
\varphi \circ \hat{\varphi} &= [N]_{E'},\end{align*} where $[N]_E$ denotes the multiplication-by-$N$ map on $E$. When $\varphi$ is cyclic, so is $\hat{\varphi}$. Since a point $\tau \in Y_0(N)$ corresponds to a cyclic $N$-isogeny $\varphi$, we conclude that there must be some other point in $Y_0(N)$ corresponding to the dual isogeny $\hat{\varphi}$.

\defn{The \emph{Fricke involution} $w_N$ on $Y_0(N)$ is the map that sends a cyclic $N$-isogeny to its dual. We sometimes write $w$ in place of $w_N$ when $N$ is clear from context.}

\begin{proposition}
	The Fricke involution is given by the map $w_N: \tau \mapsto -\frac{1}{N\tau}$.
\end{proposition}

\begin{proof} Let $\varphi$ be a cyclic $N$-isogeny $(\C/\Lambda_\tau, \wangle{1/N})$, with $\tau\in \BH$. Its image is
\begin{align*}
	\C \Big/\left(\frac{1}{N} \Z + \tau\Z\right) \cong \C/\left(\Z + N\tau\Z\right) \cong \C/(N\tau\Z - \Z) \cong \C\Big/\left(\Z - \frac{1}{N\tau}\Z\right) = \C/\Lambda_{-1/N\tau}
\end{align*}
where the isomorphisms are given by scaling by $N$, applying the change of basis $S = \begin{pmatrix}
	0 & 1\\
	-1 & 0
\end{pmatrix} \in \SL_2(\Z)$, and then scaling by $\frac{1}{N \tau}$. A point $z \in \C/\Lambda_\tau$ is mapped to $-\frac{z}{\tau} \in \C/\Lambda_{-1/N\tau}$ under this isogeny. (Note that $-\frac{1}{N\tau} \in \mathbb{H}$, justifying the otherwise superfluous negative sign.)

In the other direction, the image of the isogeny $(\C/\Lambda_{-\frac{1}{N\tau}}, \wangle{1/N})$ is 
\begin{align*}
	\C\Big/\left(-\frac{1}{N\tau} \Z + \frac{1}{N}\Z\right) \cong \Lambda_{-\frac{1}{\tau}} \cong \Lambda_\tau,
\end{align*}
and $z \in \C/\Lambda_{-1/N\tau}$ gets mapped to $-N\tau z \in \C/\Lambda_\tau$. Thus, these two isogenies compose to the multiplication-by-$N$ map, so they are dual isogenies. The isogeny $(\C/\Lambda_\tau, \wangle{1/N})$ corresponds to $\tau \in Y_0(N),$ and the dual isogeny $(\C/\Lambda_{-\frac{1}{N\tau}}, \wangle{1/N})$ corresponds to $w\tau = -\frac{1}{N\tau}$.
\end{proof}

\defn{The curve $Y_0^+(N) = Y_0(N)/\{1, w\}$ is the quotient of $Y_0(N)$ by the action of the Fricke involution; we likewise define $X_0^+(N) = X_0(N)/\{1, w\}$.}

We interpret points of $Y_0^+(N)$ as unordered pairs of curves $E, E'$ related by cyclic $N$-isogenies. If a cyclic isogeny $E \to E$ of degree $N^2 > 1$ is self-dual, then it is not given by a multiplication-by-$N$ map, since such a map has kernel $$\Z/N\Z \times \Z/N\Z \not\cong \Z/N^2\Z.$$ We therefore conclude that the ramified points of the quotient $Y_0(N) \to Y_0^+(N)$, i.e. the fixed points of the Fricke involution, give elliptic curves with complex multiplication, though not all CM curves arise in this way. 

Let $N_1N_2$ be a coprime factorization of $N$, i.e. with $N_1$ and $N_2$ coprime. Let $P \in Y_0(N)$ correspond to a cyclic $N$-isogeny $\psi: E \to E'$ with cyclic kernel $C$, so that $E' \cong E/C$. We may factor this isogeny by first forming the quotient by the unique subgroup $C_1$ of order $N_1$ of $C$, and then forming the quotient by the image of $C$ in $E/C_1$. Alternatively, we may first form the quotient by the unique subgroup $C_2$ of order $N_2$, and then form the quotient by the image of $C$ in $E/C_2$. Both of these compositions yield $\psi$:
\begin{center}
	\begin{equation}\begin{tikzcd} \label{atkin-lehner diagram}
		& E/C_1 \arrow[dr, "\phi_2"]& \\
		E \arrow[ur, "\phi_1"] \arrow[rr, "\psi"]\arrow[dr, "\varphi_2"] & & E' \cong E/C\\
		&E/C_2 \arrow[ur, "\varphi_1"] &
	\end{tikzcd}
	\end{equation}
\end{center}
Here we have labeled isogenies so that $\phi_i$ and $\varphi_i$ are cyclic $N_i$-isogenies. Using the dual isogeny, we obtain a new $N$-isogeny
\begin{align*}
	\varphi_2 \circ \hat{\phi}_1: E/C_1 \to E/C_2.
\end{align*}
The kernel of this isogeny is isomorphic to $\Z/N_1 \Z \times \Z/N_2\Z$, which is isomorphic to $\Z/N\Z$ since $N_1$ and $N_2$ are coprime. Therefore, $\varphi_2 \circ \hat{\phi}_1$ is itself a cyclic $N$-isogeny, so for each such factorization $N_1N_2 = N$ we may associate to $P$ a unique point $P(N_1)$ corresponding to the isogeny $\varphi_2 \circ \hat{\phi}_1$. This association is an involution, since applying the construction a second time transforms $\varphi_2 \circ \hat{\phi_1}$ into $\phi_2 \circ \phi_1$. We have $P(1) = P$ and $P(N) = w(P)$, but we also obtain a larger class of involutions on $Y_0(N)$. 

\defn{Let $N_1N_2$ be a coprime factorization of $N$. Then the \emph{Atkin-Lehner involution} $w_{N_1}$ associated to $N_1$ on $Y_{0}(N)$ is the map
\begin{align*}
	Y_{0}(N) \to Y_0(N): P \mapsto P(N_1)
\end{align*}
as described above. We let $W(N)$ denote the group generated by the Atkin-Lehner involutions, and we let and $Y^*(N)$ denote the quotient of $Y_0(N)$ by the action of $W(N)$.}

We cite some results about the Atkin-Lehner involutions from \cite{lmfdb}. Generalizing the realization of the Fricke involution as the map $\tau \mapsto -1/N\tau$, we also have explicit representations of the Atkin-Lehner involutions as M\"obius transformations. 
\begin{proposition}
	The Atkin-Lehner involution $w_{N_1}$ is given as a M\"obius transformation by any matrix of form
	\begin{align*}
	 	\begin{pmatrix}
	 		aN_1 & b\\
	 		cN & dN_1
	 	\end{pmatrix}
	 \end{align*} 
	 with determinant $N_1$. 
\end{proposition}

This representation allows us to define an action of $W(N)$ on $X_0(N)$ as well, and thus allows us to define the quotient space $X^*(N) = X_0(N)/W(N)$.

\begin{proposition}
	If $N_1, M_1$ are coprime divisors of $N$, then $$w_{N_1} w_{M_1} = w_{N_1M_1}.$$ In particular, the Atkin-Lehner involutions on $X_0(N)$ commute.
\end{proposition}
\begin{corollary}
	Let $\omega(N)$ denote the number of distinct prime divisors of $N$. Then $W(N) \cong (\Z/2\Z)^{\omega(N)}$, with each element of this group corresponding to an ordered coprime factorization $N_1N_2 = N$. 
\end{corollary}

\subsection{Modular functions and modular forms}
In this section we cite several standard results on modular forms. For more details, see \cite{MR2112196}. Throughout, we let $q(\tau) := e^{2\pi i \tau}$.

\defn{For even $k \geq 4$, the \emph{normalized Eisenstein series of weight $k$} is
\begin{align*}
	\mathsf{E}_k(\tau) := 1 - \frac{2k}{B_{k}} \sum_{n = 1}^\infty \sigma_{k- 1}(n)q^n,
\end{align*}
where $\sigma_{k}(n) = \sum_{d\mid n} d^k$ denotes the sum-of-divisors function and $B_k$ denotes the $k$-th Bernoulli number: $$B_0 = 1, \; B_{1} = -\frac{1}{2}, \; B_2 = \frac{1}{6}, \; B_4 = -\frac{1}{30}, \; B_6 = \frac{1}{42}, \; B_8 = -\frac{1}{30}, \; B_{10} = \frac{5}{66}, \; B_{12} = -\frac{691}{2730}, \dots$$ (The odd-indexed Bernoulli numbers $B_{2n + 1}$ vanish for $n > 0$.) We also define a second normalization $G_k$ of the Eisenstein series given by
\begin{align*}
	G_{k}(\tau) := 2\zeta(k) \mathsf{E}_{k}(\tau)
\end{align*}
where $\zeta$ is the Riemann zeta function.}

Both normalizations of the Eisenstein series are useful: the $q$-expansion of $\mathsf{E}_k$ has rational coefficients involving the arithmetic function $\sigma_{k - 1}$, while the series $G_{k}$ are involved in the correspondence between complex elliptic curves $\C/\Lambda$ and their affine models:  

\begin{theorem}\label{Uniformization}
	 \emph{(Uniformization.)} Let $\C/\Lambda_\tau$ be a complex elliptic curve. Then 
	\begin{align*}
		E: y^2 = 4x^3 - 60G_4(\tau)x - 140G_6(\tau)
	\end{align*}
	is an affine elliptic curve, and the map $z \mapsto (\wp(z), \wp'(z))$ is an isomorphism of complex Lie groups. Moreover, for every affine elliptic curve $F/\C$ defined by a Weierstrass equation as above, there exists a unique lattice $\Lambda_\tau$ up to homothety such that the curve $E$ given above is isomorphic to $F$.
\end{theorem}
\begin{proof}
	See [\cite{MR2514094}, Proposition VI.3.6 and Corollary VI.5.1.1]. 
\end{proof}
\begin{proposition}
	The normalized Eisenstein series $\mathsf{E}_k$ is a modular form of weight $k$ on $X(1)$; that is,
	\begin{align*}
		\mathsf{E}_k\left(\frac{a\tau + b}{c\tau + d}\right) = (c\tau + d)^k \mathsf{E}_k(\tau)
	\end{align*}
	for any $\begin{pmatrix}
		a & b \\ c & d
	\end{pmatrix} \in \SL_2(\Z)$.
\end{proposition}

\defn{The \emph{Dedekind eta function} is 
\begin{align*}
	\eta(\tau) := q^{1/24} \prod_{n = 1}^\infty (1 - q^n).
\end{align*}
Here $q^{1/24}$ denotes $\exp(\pi i \tau/12)$.} 

\begin{proposition}
	The eta function satisfies the following functional equations:
\begin{align*}
	\eta(\tau + 1) &= \exp(\pi i/12)\; \eta(\tau),\\
	\eta(-1/\tau) &= \sqrt{-i\tau} \; \eta(\tau),
\end{align*}
where we take the branch of the square root that agrees with the square root on the positive real numbers. 
\end{proposition}

\begin{lemma}{\label{nu-lemma}}
	Let $\nu(\tau) := \eta(N\tau)/\eta(\tau)$.
	\begin{enumerate}[(a)]
		\item The function $\nu(\tau)$ transforms under the Fricke involution as
		\begin{align*}
		w_N^* \nu(\tau) = \nu(-1/N\tau) = \frac{1}{\sqrt{N} \nu(\tau)}.
		\end{align*}
		\item The derivative $\nu'(\tau)$ transforms under the Fricke involution as
		\begin{align*}
			w_N^* \nu(\tau) = \nu'(-1/N\tau) = -\sqrt{N}\tau^2 \cdot \frac{\nu'(\tau)}{\nu(\tau)^2}.
		\end{align*}
	\end{enumerate}

\end{lemma}
\begin{proof}
	Part (a) follows immediately from the functional equation $\eta(-1/\tau) = \sqrt{-i\tau}\eta(\tau)$. For part (b), apply the chain rule to write
	\begin{align*}
		\nu'(-1/N\tau) \cdot \frac{1}{N\tau^2} &= \frac{d}{d\tau}\nu(-1/N\tau).
	\end{align*}
	Substituting using part (a), this gives
	\begin{align*}
		\nu'(-1/N\tau) &= N\tau^2 \frac{d}{d\tau}\left( \frac{1}{\sqrt{N} \nu(\tau)}\right)\\
		&= -\sqrt{N} \tau^2 \cdot \frac{\nu'(\tau)}{\nu(\tau)^2}.
	\end{align*}
\end{proof}
\defn{
	The \emph{$j$-invariant} is
	\begin{align*}
		j(\tau) := 1728 \cdot \frac{\mathsf{E}_2(\tau)^3}{\mathsf{E}_2(\tau)^3 - 27\mathsf{E}_3(\tau)^2}.
	\end{align*}
}
\defn{\label{haupt-def} A \emph{Hauptmodul} $h$ (plural \emph{Hauptmoduln}) of a modular curve $X$ of genus $0$ is a generator of the function field $\C(X) = \C(h)$ normalized to have $q$-expansion $\frac{1}{q} + O(1)$. In particular, $h$ has its pole at the cusp at $\infty$.}
\begin{proposition}
	The $j$-invariant is a modular function on $X(1)$. Moreover, it is a Hauptmodul for $X(1)$, so that it parameterizes $\C$-isomorphism classes of elliptic curves. 
\end{proposition}
\begin{proof}
	The $j$-invariant is a modular function on $X(1)$ since it is the quotient of two weight $12$ modular forms and thus has trivial transformation law under $\SL_2(\Z)$. The fact that it is a Hauptmodul follows from computing its $q$-expansion 
	$$j(q) = q^{-1} + 744 + 196884q + \dots$$
	and from the fact that $j$ parameterizes $\C$-isomorphism classes of elliptic curves; see [\cite{MR2514094}, Proposition III.1.4]. 
\end{proof}

\defn{The \emph{Eisenstein series of weight $2$ and level $N$} is
\begin{align*}
	\mathsf{E}_2^{(N)}(\tau) := q \cdot \frac{d}{dq} \log \left(\frac{\eta(q^N)}{\eta(q)}\right) = \frac{N - 1}{24} + \sum_{n = 1}^\infty \sigma_1(n)(q^n - Nq^{Nn}).
\end{align*}
}
\begin{proposition}
	The Eisenstein series $\mathsf{E}_2^{(N)}(\tau)$ is a modular form of weight $2$ on $X_0(N)$. 
\end{proposition}
\begin{proposition}\label{E2-antiinvariant}
	The modular form $\mathsf{E}_2^{(N)}(\tau)$ is anti-invariant under the Fricke involution. That is, it transforms under the Fricke involution by the functional equation 
	\begin{align*}
		\mathsf{E}_2^{(N)}(-1/N\tau) = -N\tau^2E_2^{(N)}(\tau).
	\end{align*}
\end{proposition}
\begin{proof}
	Write $\tau = \frac{\log q}{2\pi i}$. Then by the chain rule $\frac{d}{dq} = \frac{1}{2\pi i q} \frac{d}{d\tau}$ as differential operators. We use this to define $\mathsf{E}_2^{(N)}(\tau)$ entirely in terms of $\tau$, eliminating the variable $q$:
	\begin{align*}
		\mathsf{E}_2^{(N)}(\tau) &= \frac{1}{2\pi i} \frac{d}{d\tau} \log\left(\frac{\eta(N\tau)}{\eta(\tau)}\right).
	\end{align*}
	Write $\nu(\tau) = \frac{\eta(N\tau)}{\eta(\tau)}$. Taking logarithmic derivatives, this is
	\begin{align*}
		\mathsf{E}_2^{(N)}(\tau) &= \frac{1}{2\pi i} \frac{\nu'(\tau)}{\nu(\tau)}.
	\end{align*}
Apply Lemma \ref{nu-lemma} to get
	\begin{align*}
		w_N^*\frac{\nu'(\tau)}{\nu(\tau)} &= -\frac{\sqrt{N}\tau^2\frac{\nu'(\tau)}{\nu(\tau)^2}}{\frac{1}{\sqrt{N} \nu(\tau)}}\\
		&= -N \tau^2 \cdot \frac{\nu'(\tau)}{\nu(\tau)},
	\end{align*}
	and we conclude
	\begin{align*}
		w_N^* \mathsf{E}_2^{(N)}(\tau) = -N\tau^2 \mathsf{E}_2^{(N)}.
	\end{align*}

\end{proof}

\subsection{$K$-rational points on modular curves} \label{K-rational}

When defining elliptic curves and isogenies algebraically, we have a notion of a field of definition $K$. Since modular curves parameterize elliptic curves, we would like to similarly have a notion of $K$-rational functions and $K$-rational points that correspond to enhanced elliptic curves defined over $K$. We focus only on the modular curves $X_0(N)$, since these are our primary interest.

\defn{The set of \emph{$K$-rational functions} $K(X_0(N))$ on $X_0(N)$ consists of the modular functions on $X_0(N)$ having $q$-expansions with coefficients in $K$.}

\defn{The set of \emph{$K$-rational points $X_0(N)(K)$} on $X_0(N)$ are those corresponding to cyclic $N$-isogenies defined over $K$.}

It is not obvious that these two definitions are compatible; we would rather treat $X_0(N)$ as a projective curve in order to translate the transcendental language of $q$-expansions to the rational language of function fields in algebraic geometry. Fortunately, this is always possible.   
\begin{theorem}{\emph{[\cite{stepanov}, Theorem 8.10].}}
There exists an irreducible integer polynomial $\Phi_N(x, y)$, known as the modular polynomial, with vanishing set parameterized by $x = j(\tau), y = j(N\tau)$ over $\tau \in X_0(N)$.
\end{theorem} 

This is similar to the map between a complex elliptic curve $\C/\Lambda$ to an algebraic elliptic curve in the sense that it translates between a Riemann surface and an affine model. The affine version of the modular curve $X_0(N)$ is sometimes called the ``classical modular curve.'' Every point $\tau \in Y_0(N)$ corresponding to a $K$-rational cyclic $N$-isogeny $E \to E'$ maps to a $K$-rational point on the classical modular curve, since $j(\tau) = j(E) \in K$ and $j(N\tau) = j(-1/N\tau) = j(E') \in K$. 

Note that the Fricke involution $w_N^*$ swaps the coordinates $x$ and $y$, so that it is always a rational automorphism of order $2$ on the classical modular curve over any field of definition. It can be shown that the other Atkin-Lehner involutions are also rational automorphisms of order $2$. Thus, when forming the spaces $X_0^+(N) = X_0(N)/\{1, w_N\}$ and $X^*(N) = X_0(N)/W(N)$, we may treat these quotient spaces algebraically. This verifies that the natural maps $X_0(N) \to X_0^+$ and $X_0(N) \to X^*(N)$ are $K$-rational for any field $K$, since the automorphisms defining the quotient are $K$-rational. A $K$-rational point on $X_0^+(N)$ must have fiber $\mathcal{P} \subset X_0(N)$ stable under $\Gal(\conj{K}/K)$. Since the map $X_0(N) \to X_0^+$ has degree $2$, we conclude that either $\mathcal{P}$ consists of one or two points defined over $K$, or it consists of two points $P, P'$ defined over some quadratic extension $L/K$ with $P, P'$ related by Galois conjugation. Similarly, the fiber in $X_0(N)$ of a $K$-rational point on $X^*(N)$ consists of points lying in some multi-quadratic extension of $K$ since $W(N) \cong (\Z/2\Z)^{\omega(N)}$, and this fiber is stable under the Galois action $\Gal(\conj{K}/K)$; we will return to this idea in our discussion of $K$-curves. 

One issue with treating $X_0(N)$ as a curve over $K$ is that points of $X_0(N)$ correspond to cyclic $N$-isogenies up to $\C$-isomorphism, \emph{not} up to $K$-isomorphism. In general, a $\C$-isomorphism class of isogenies will split into infinitely many $K$-isomorphism classes of isogenies related by twists; as long as we avoid $j \in \{0, 1728\}$, these will be quadratic twists. Therefore, to a given $K$-rational point $\tau \in X_0(N)$, there is no canonical association with a particular $K$-isomorphism class of affine elliptic curves $E \to E'$ related by a cyclic $N$-isogeny in the sense of Theorem \ref{modular-correspondence}. However, once a choice of twist for $E$ is fixed, this determines the twist of $E'$ and the isogeny. 

\subsection{The Tate curve}\label{Tate-curve-section}

Some computations involving modular curves and modular forms can be reduced to linear algebra on $q$-expansions. For example, this is one way to prove the Uniformization Theorem \ref{Uniformization}. It is therefore useful when considering the model of an affine elliptic curve to treat it as a curve with coefficients given by $q$-series; the Tate curve does precisely this. 

Let $K$ be a number field. Instead of treating an elliptic curve $E/K$ as a projective plane curve with coefficients in $K$, we may instead apply the Uniformization Theorem to treat all elliptic curves over number fields as specializations of a plane curve defined over $\C((q))$, known as the \emph{Tate curve}: 
\begin{align*}
 	\mathbb{G}/q^{\Z}: y^2 = 4x^3 - 60G_4(q) x - 140G_6(q).
\end{align*}
We may apply a change of variables to transform the Eisenstein series $G_4(q), G_6(q)$ into their normalized versions $\mathsf{E}_4(q), \mathsf{E}_6(q)$ to show that $\mathbb{G}/q^{\Z}$ is actually defined over $\Q((q))$, with model
\begin{align*}
	\mathbb{G}/q^{\Z}: y^2 = x^3 - \frac{\mathsf{E}_4}{48}x + \frac{\mathsf{E}_6}{864}. 
\end{align*}
See [\cite{MR1486831}, Equation (32)].

One benefit of viewing elliptic curves from the perspective of the Tate curve is that the Weierstrass coefficients can be treated as modular functions. For any modular curve $X$, the coefficients $-\frac{\mathsf{E}_4}{48}$ and $\frac{\mathsf{E}_6}{864}$ are modular forms on $X$ of weight $4$ and $6$ respectively. However, if $\lambda$ is a modular form of weight $2$ on $X$, then the change of variables $(x, y) \mapsto (\lambda x, \lambda^{3/2}y)$ twists the Tate curve into the curve 
\begin{align}\label{critical-interpretation}
	\mathbb{G}/q^{\Z}: y^2 = x^3 - \frac{\mathsf{E}_4}{48\lambda^2}x + \frac{\mathsf{E}_6}{864 \lambda^3}.
\end{align}
 Now the coefficients
\begin{align*}
 	a_4 = -\frac{\mathsf{E}_4}{48\lambda^2}, \;\;\;\;\;\;\; a_6 = \frac{\mathsf{E}_6}{864\lambda^3}
 \end{align*} 
may be treated as modular \emph{functions} on $X$. In this model, the coordinate functions $x, y$ may be given as functions of $\tau$ and $z \in \C/\Lambda_\tau$ based on the Weierstrass $\wp$-function, and indeed as functions of $q = \exp(2\pi i \tau)$ and $q_z = \exp(2\pi i z)$. More precisely, we have
\begin{align}\label{x-formula}
	x &= \lambda^{-1} \left[\frac{1}{2} - 2\sum_{n = 1}^\infty \frac{q^n}{(1 - q^n)^2} + \sum_{n = -\infty}^\infty \frac{q^n q_z}{(1 - q^nq_z)^2}\right],\\
	\label{y-formula}
	y &= \frac{1}{2} \lambda^{-3/2} \sum_{n = -\infty}^\infty \frac{(q^nq_z)^2 + q^nq_z}{(1 - q^nq_z)^3}.
\end{align}
Again, see [\cite{MR1486831}, Equation (31)]. 

 We will discuss explicit parameterizations of the function field $K(X_0(N))$ in Section \ref{section-3} when the modular curve $X_0(N)$ has genus $0$, and so we will be able to explicitly identify the coefficients $a_4, a_6$ in the model (\ref{critical-interpretation}) in terms of this parameterization. Taking $\lambda = \mathsf{E}_2^{(N)}$ will be a useful choice of weight $2$ modular form on $X_0(N)$ since $\mathsf{E}_2^{(N)}$ is anti-invariant under the Fricke involution by Proposition \ref{E2-antiinvariant}. This model will play a critical role in the proof of the Main Theorem \ref{main-theorem} on strict $K$-curves over fibers of $X_0^+(N)$. 

\section{Rational parameterizations of genus 0 modular curves}\label{section-3}
\subsection{A Hauptmodul of $X_0(N)$ as an eta product}
Modular curves have parameterizations that give rational expressions for various quantities of interest, such as the $j$-invariant or the coefficients of the corresponding elliptic curve. This is especially true for modular curves of genus $0$: rational functions involving Hauptmoduln are effective computational tools. 

\begin{proposition}\label{N-genus-0}
	The modular curve $X_0(N)$ has genus $0$ if and only if $N$ is one of the following $15$ integers: $$1, 2, 3, 4, 5, 6, 7, 8, 9, 10, 12, 13, 16, 18, 25.$$ 
\end{proposition}
\begin{proof}
	One can prove this by using the formula [\cite{MR2112196}, Theorem 3.1.1] for the genus of a modular curve $X(\Gamma)$
	\begin{align*}
		g = 1 + \frac{d}{12} - \frac{\varepsilon_2}{4} - \frac{\varepsilon_3}{3} - \frac{\varepsilon_\infty}{2},
	\end{align*}
	where $d$ is the degree of the map $X(\Gamma) \to X(1)$, the quantities $\varepsilon_2$ and $\varepsilon_3$ are the number of elliptic points of period $2$ and $3$ on $X(\Gamma)$, and $\varepsilon_\infty$ is the number of cusps on $X(\Gamma)$. In the case $X(\Gamma) = X_0(N)$, we have
	\begin{align*}
		d &= N\prod_{p \mid N} \left(1 + \frac{1}{p}\right),\\
		\varepsilon_2 &= \begin{cases}
			\prod_{p \mid N} \left(1 + \left(\frac{-1}{p}\right)\right): &4\nmid N\\
			0 : &4\mid N,
		\end{cases}\\
		\varepsilon_3 &= \begin{cases}
			\prod_{p \mid N} \left(1 + \left(\frac{-3}{p}\right)\right): &9\nmid N\\
			0: &9 \mid N,
		\end{cases}\\
		\varepsilon_\infty &= \sum_{\delta \mid N} \varphi(\gcd(\delta, N/\delta)),
	\end{align*}
	where 
	\begin{itemize}
		\item $p$ denotes a prime divisor of $N$;
		\item  $\delta$ denotes an arbitrary positive divisor of $N$;
		\item $\left(\frac{-1}{p}\right)$ and $\left(\frac{-3}{p}\right)$ are Legendre symbols, with $ \left(\frac{-3}{3}\right)$ defined to be $0$; and
		\item $\varphi$ is the Euler totient function.
	\end{itemize}
	See [\cite{MR2112196}, Figure 3.3]. 
\end{proof}

The curve $X_0(1) = X(1)$ has the $j$-invariant as a Hauptmodul. We would like to explicitly compute a Hauptmodul for the other $14$ values of $N$. Eta products provide a convenient method of doing so; these are products of the form
\begin{align*}
	 \prod_{\delta \mid N} \eta(\delta \tau)^{r_\delta}
\end{align*}
with each $r_\delta$ an integer. In order for an eta product to be a Hauptmodul, we recall the normalization conditions from Definition \ref{haupt-def} and require that
\begin{enumerate}
 	\item The Hauptmodul $h_N$ has a simple zero at the cusp $\tau = 0$ and a simple pole at the cusp $\tau = \infty$; and
 	\item The $q$-expansion of $h_N$ is of the form $$1 \cdot q^{-1} + O(1).$$
\end{enumerate}

Note that the first condition is possible because $0$ is always inequivalent to $\infty$ under $\Gamma_0(N)$ as long as $N > 1$. Some authors, such as \cite{MR2514149}, swap our $h_N$ with $w_N^*h_{N}$; under our definitions, the function $w_N^*h_{N}$ is not a Hauptmodul since it does not have its pole at $\infty$, but it is an alternative rational coordinate on $X_0(N)$. Some authors also refer to ``eta products'' as ``eta quotients.''

We regard the eta function as vanishing to order $1/24$ at $\infty$ since $\eta^{24}$ is a constant multiple of the modular discriminant $\Delta$, which has a simple zero at $\infty$. Additionally, since $\Delta$ has no zeros or poles on $\mathbb{H}$, neither does the eta function. This leads us to the \emph{Ligozat conditions}, which give number-theoretic conditions on the exponents $r_\delta$ to ensure that the eta product is a modular function:

\begin{proposition}
	\emph{[\cite{MR0422158}, Proposition 3.2.1].} The function $\prod_{\delta \mid N} \eta(\delta \tau)^{r_\delta}$ defines a modular function on $X_0(N)$ if and only if the following four conditions hold:
	\begin{enumerate}[(a)]
		\item $\sum_{\delta \mid N} r_\delta \cdot \delta \equiv 0\mod 24,$
		\item $\sum_{\delta \mid N} r_\delta \cdot (N/\delta) \equiv 0 \mod 24,$
		\item $\sum_{\delta \mid N} r_\delta = 0,$
		\item $\prod_{\delta \mid N} (N/\delta)^{r_\delta} \in (\Q^\times)^2$.
	\end{enumerate}

\end{proposition}

Condition (a) corresponds to integrality of the order of vanishing at $\infty$, and condition (b) corresponds to integrality at $0$. However, we require more than integrality: we need the order at $\infty$ to be exactly $-1$ and the order at $0$ to be exactly $+1$. We also require that the eta product have zero order of vanishing at all other cusps. Let $\frac{a}{d} \in \Q$ be in lowest form. To find the order of vanishing of $\eta(\delta \tau)$ at $\frac{a}{d}$, we apply some ``cusp-normalizing map'' $\gamma \in \SL_2(\Z)$ that sends $\infty \mapsto \frac{a}{d}$, and we may use the functional equation of $\eta$ to verify that $\eta(\delta \tau)$ has order $\frac{1}{24} (\delta, d)^2/\delta$ at $\frac{a}{d}$. Here $(\cdot, \cdot)$ denotes the greatest common divisor.

To compute the order at the cusp $\left[\frac{a}{d}\right] \in X_0(N)$, we must multiply by the width $(N/d) \cdot (d, N/d)$ of this cusp. All cusps of $X_0(N)$ have a representative $\frac{a}{d}$ with $d$ a divisor of $N$, so this means that it suffices to impose the conditions
 \begin{align*}
 	(N/d) (d, N/d) \cdot \sum_{\delta \mid N}  r_\delta \cdot (\delta, d)^2/\delta = 0
 \end{align*}
 for each divisor $d \mid N$ other than $d = 1, N$. We include the factor of $(N/d)(d, N/d)$ so that the coefficient of each $r_\delta$ is an integer, but omitting it does not affect the condition. The divisors $1$ and $N$ correspond to the cusps at $0 \in \left[\frac{1}{1}\right]$ and $\infty \in \left[\frac{1}{N}\right]$ respectively, for which we have the conditions
 \begin{align*}
 	\sum_{\delta \mid N}  r_\delta \cdot (N/\delta) &= 24,\\
 	 \sum_{\delta \mid N}  r_\delta \cdot \delta &= -24.
 \end{align*} 
for order of vanishing $1$ and $-1$ respectively. We summarize this discussion as a proposition: 
\begin{proposition}\label{Hauptmodul-condition}
	An eta product $\prod_{\delta \mid N} \eta(\delta \tau)^{r_\delta}$ defines a Hauptmodul on the modular curve $X_0(N)$ if and only if the following conditions hold on the integer exponents $r_\delta$:
	\begin{enumerate}[(a)]
		\item $\sum_{\delta \mid N} r_\delta \cdot \delta = -24,$
		\item $
		\sum_{\delta \mid N} r_\delta \cdot (N/\delta) = 24,$
		\item $
		\sum_{\delta \mid N} r_\delta = 0,$
		\item 
		$\prod_{\delta \mid N} (N/\delta)^{r_\delta} \in (\Q^\times)^2,$
	\end{enumerate}
	and for all divisors $d \mid N$ besides $d = 1, N$, we have
	\begin{enumerate}[(e)]
		\item $
		(N/d) (d, N/d) \sum_{\delta \mid N}  r_\delta \cdot (\delta, d)^2/\delta = 0.$
	\end{enumerate}
\end{proposition}

The linear conditions (a), (b), (c), and (e) in Proposition \ref{Hauptmodul-condition} are enough to uniquely specify the exponents $r_\delta$ when $N > 1$ is in the list given in Proposition \ref{N-genus-0}, and we may check that these solutions also satisfy condition (d). We denote the corresponding eta product by $h_N$, or by $h$ when the value of $N$ is clear. These eta products are listed in Table \ref{hauptmoduln}. In this table, we also describe the image of these Hauptmoduln under the Fricke involution. We always have $w_N^*h_N = \kappa_N/h_N$ for some constant $\kappa_N$, which can be readily verified using the functional equation for $\eta$.
\renewcommand{\arraystretch}{2.5}
\begin{figure}[h]
\begin{tabular}{|c||c|c|c|c|c|c|c|c|}
	\hline $N$ & 2 & 3 & 4 & 5 & 6 & 7 & 8 & 9\\
	\hline
	$h_N$ & $\left(\frac{\eta(\tau)}{\eta(2\tau)}\right)^{24}$ & $\left(\frac{\eta(\tau)}{\eta(3\tau)}\right)^{12}$ & $\left(\frac{\eta(\tau)}{\eta(4\tau)}\right)^8$ &  $ \left(\frac{\eta(\tau)}{\eta(5\tau)}\right)^{6}$ & $\frac{\eta^5(\tau) \eta(3\tau)}{\eta(2\tau)\eta^5(6\tau)}$ & $\left(\frac{\eta(\tau)}{\eta(7\tau)}\right)^4$ & $\frac{\eta^4(\tau)\eta^2(4\tau)}{\eta^2(2\tau)\eta^4(8\tau)}$ &$\left(\frac{\eta(\tau)}{\eta(9\tau)}\right)^3$\\
	\hline
	$w_N^*h_N$ & $2^{12}/h_2$ & $3^{6}/h_3$ & $2^{8}/h_4$ & $5^{3}/h_5$ & $2^3 \cdot 3^2 /h_6$ & $7^2/h_7$ & $2^5/h_8$ & $3^3/h_9$\\
	\hline
\end{tabular}

\begin{tabular}{|c||c|c|c|c|c|c|}
	\hline$N$ & 10 & 12 & 13 & 16 & 18 & 25 \\
	\hline
	$h_N$ & $\frac{\eta^3(\tau)\eta(5\tau)}{\eta(2\tau)\eta^3(10\tau)}$ & $\frac{\eta^3(\tau)\eta(4\tau) \eta^2(6\tau)}{\eta^2(2\tau)\eta(3\tau)\eta^3(12\tau)}$ & $\left(\frac{\eta(\tau)}{\eta(13\tau)}\right)^{2}$ & $\frac{\eta^2(\tau)\eta(8\tau)}{\eta(2\tau)\eta^2(16\tau)}$ & $\frac{\eta^2(\tau)\eta(6\tau)\eta(9\tau)}{\eta(2\tau)\eta(3\tau)\eta^2(18\tau)}$ & $\frac{\eta(\tau)}{\eta(25\tau)}$\\
	\hline $w^*_Nh_N$ & $20/h_{10}$ & $12/h_{12}$ & $13/h_{13}$ & $8/h_{16}$ & $6/h_{18}$ & $5/h_{25}$\\
	\hline
\end{tabular}
\caption{The Hauptmodul $h_N$ of $X_0(N)$ as an eta product and its image $w_N^*h_N$ under the Fricke involution.}
\label{hauptmoduln}
\end{figure}

\subsection{The $j$-invariant as a rational function of each Hauptmodul}

The degree of the map $X_0(N) \to X(1)$ is equal to the index of $\Gamma_0(N)/\{\pm I\}$ in $\SL_{2}(\Z)/\{\pm I\}$, which is
\begin{align*}
	\psi(N) := N \prod_{p | N} \left(1 + \frac{1}{p}\right). 
\end{align*}
Therefore, when $X_0(N)$ has genus $0$, there exists a degree $\psi(N)$ rational function giving the $j$-invariant in terms of the Hauptmodul for $X_0(N)$. Since the degree of the rational function is known, this expression may be found by using linear algebra to compare $q$-expansions. Explicitly, we have the following algorithm:
\renewcommand{\arraystretch}{1}
\begin{enumerate}
	\item Let $n = \psi(N)$, let $h = q^{-1} + O(1)$ be the Hauptmodul of $X_0(N)$, and let $H = 1/h$. Compute the inverse $$Q(H) + O(H^{2n + 3}) = q(H)$$ of the $q$-expansion of $H$ to at least $2n + 2$ terms. 
	\item Compute $$J(H) + O(H^{2n + 2}) := j(Q(H))$$ to express $j$ as a Laurent series $J(H) = H^{-1} + O(1)$ in $H$ correct to a precision of $2n + 1$ terms. Let the coefficient of $H^i$ in $J(H)$ be $a_i$. 
	\item Find a generator $(b_{n}, b_{n - 1}, \dots, b_0)$ of the $1$-dimensional kernel of the $(n + 2)\times (n + 2)$ matrix
	\begin{align*}
	\begin{pmatrix}
		1 & a_0 & a_1 &\dots & a_{n}\\
		a_0 & a_1 & a_2 &\dots & a_{n + 1}\\
		a_1 & a_2 & a_3 &\dots & a_{n + 2}\\
		\vdots & \vdots &\vdots & \ddots & \vdots\\
		a_{n} & a_{n + 1} &  a_{n + 2} & \dots & a_{2n + 1}
	\end{pmatrix}
		\end{align*}
	\item Compute 
	$$P(H) = J(H) \cdot (b_{n}H^{n} + b_{n - 1}H^{n + 1} + \dots + b_0).$$
	Then 
	\begin{align*}
		j = \frac{P(H)}{b_{n}H^{n} + b_{n - 1}H^{n + 1} + \dots + b_0}.
	\end{align*}
	 Substitute $1/h  = H$ to obtain $j$ as a rational function of $h$.
\end{enumerate}
As a ``sanity check,'' in practice we compute the series in the algorithm to a few more terms than required. We list in Table \ref{j-invar} the results up to $N = 5$; see [\cite{MR2514149}, \S 3.2] for a full table.
\renewcommand{\arraystretch}{2.5}
\begin{figure}[h]
\begin{center}
	\begin{tabular}{|c || c | c|}
		\hline $N$ & $j$ & $j - 1728$\\
		\hline $2$ & $\frac{(h + 256)^3}{h^2}$ & $\frac{(h + 64)(h - 512)^2}{h^2}$ \\
		\hline $3$ & $\frac{(h + 27)(h + 243)^3}{h^3}$ & $\frac{(h^2 - 486h - 19683)^2}{h^3}$ \\
		\hline $4$ & $\frac{(h^2 + 256h + 4096)^3}{h^4(h + 16)}$ & $\frac{(h + 32)^2 (h^2 - 512 - 8192)^2}{h^4(h + 16)}$\\
		\hline $5$ & $\frac{(h^2 + 250 h + 3125)^3}{h^5}$ & $\frac{(h^2 + 22h + 125)(h^2 - 500h - 15625)^2}{h^5}$ \\
		\hline
		\vdots & \vdots & \vdots
	\end{tabular}
	\end{center}
	\caption{Expressions of $j$ and $j - 1728$ in terms of the Hauptmodul $h = h_N$ up to $N = 5$.}
	\label{j-invar}
\end{figure}

The group $\SL_2(\Z)/\{\pm I\}$ acts freely on $\mathbb{H}^*$ except on the cusps and the \emph{elliptic points} given by the orbits of $e^{2\pi i/3}$ and $i$ corresponding to $j = 0$ and $j = 1728$ respectively. This implies that the only ramified values of the map $X_0(N) \to X(1)$ are given by $j = 0, 1728,$ and $\infty$. One benefit of computing the rational functions giving $j$ in terms of $h_N$ is that we may view the ramification behavior at $j = 0$ and $\infty$ at a glance. Similarly, we may view the ramification behavior at $j = 1728$ and $\infty$ from the rational function expressing $j - 1728$ in terms of $h_N$. For example, the expressions for $j$ and $j - 1728$ in terms of $h_2$ show that the ramification divisor of $X_0(2) \to X(1)$ is
\begin{align*}
	2[-256] + [512] + [0].
\end{align*}
with coordinates given by $h_2$. This agrees with the Riemann-Hurwitz formula for a degree $3$ map between genus $0$ Riemann surfaces.

As noted in Section \ref{K-rational}, the fixed points of the Fricke involution yield elliptic curves with complex multiplication. For each of the curves $X_0(N)$ of genus $0$, we found that $w_N^*h_N = \kappa_N/h_N$ for some constant $\kappa_N$ coming from the functional equation of the eta function, so we may immediately find the values of the Hauptmodul at the two fixed points. However, if we can independently determine these fixed points as elements $\tau \in \mathbb{H}$, we obtain the value of $h_N(\tau)$ explicitly at such $\tau$. In conjunction with the rational expression for the $j$-invariant, this allows us to find many special values of $j$ at certain CM points. 

The point $\tau = i/\sqrt{N}$ is always a fixed point of the Fricke involution $w_N$. We also know that $h_N \cdot w_N^*h_N = \kappa_N$ for some constant $\kappa_N$, so the fixed points in $h$-coordinates are given by $\pm \sqrt{\kappa_N}$.  When $\tau$ lies on the positive imaginary axis, the variable $q = \exp(2\pi i \tau)$ lies in the interval $(0, 1)$. Therefore the product $\prod_{n = 1}^\infty (1 - q^n)$ used in the definition of the eta function must be positive for $\tau$ on the positive imaginary axis. When writing a Hauptmodul as an eta product, the factors of $q^{1/24}$ from the eta function must cancel, so we conclude that $h_N$ is also positive when $\tau$ lies on the positive imaginary axis. Combining this with the expressions for the $j$-invariant from Table \ref{j-invar} (and its completion in [\cite{MR2514149}, \S 3.2]), we obtain $14$ special values of $j(\tau)$, which we list in Table \ref{special-j} in Appendix \ref{j-appendix}.

\subsection{Coefficients of elliptic curves as rational functions of Hauptmoduln}\label{section3.3}
Recall from our discussion of the Tate curve in Section \ref{Tate-curve-section} that we may explicitly find an affine model of the elliptic curve $E$ corresponding to $\tau \in X_0(N)$:
\begin{align} \label{E-model}
	E: y^2 = x^3 - \frac{\mathsf{E}_4(\tau)}{48 \lambda^2} + \frac{\mathsf{E}_6(\tau)}{864 \lambda^2}.
\end{align}
Here $\lambda$ is an arbitrary weight 2 modular form on $X_0(N)$. However, if we want to find the isogenous curve $E'$ associated to $\tau$, we must exercise caution. If $E$ is defined over $K$, then we would like the isogenous curve $E'$ to be isogenous to $E$ \emph{over} $K$ rather than merely over $\C$. Since $w_N(\tau) = -1/N\tau$, and $j(-1/N\tau) = j(N\tau)$, the isogenous curve $E'$ will have some model isomorphic to 
\begin{align} \label{Isogenous-model}
	E' : y^2 = x^3 - \frac{\mathsf{E}_4(N\tau)}{48 \lambda^2} + \frac{\mathsf{E}_6(N\tau)}{864 \lambda^3}
\end{align}
where we replace $\tau$ with $N\tau$ in the arguments of the normalized Eisenstein series, but not in the modular form $\lambda$. To prove that this is the correct twist corresponding to the chosen model (\ref{E-model}) of $E$, we note that by Theorem \ref{modular-correspondence}, the isogeny of complex curves corresponding to $\tau \in Y_0(N)$ is given explicitly by \begin{align*}\C/\Lambda_\tau &\to \C/\Lambda_{N\tau}\\
 z &\mapsto Nz.\end{align*} 
Therefore, on the Tate curve, the isogeny given on the coordinates $x, y$ by replacing $q$ with $q^N$ and $q_z$ with $q^{N}_z$ in formulas (\ref{x-formula}) and (\ref{y-formula}), and indeed these new coordinates lie on the model of $E'$ given by (\ref{Isogenous-model}). 

Taking $\lambda = \mathsf{E}_2^{(N)}(\tau)$, we can find an expression for the modular functions 
\begin{align*}
a_4 &:= -\frac{\mathsf{E}_4(\tau)}{48 \mathsf{E}_2^{(N)}(\tau)^2}, & a_6 &:= \frac{\mathsf{E}_3(\tau)}{864 \mathsf{E}_2^{(N)}(\tau)^3},\\
a'_4 &:= -\frac{\mathsf{E}_4(N\tau)}{48 \mathsf{E}_2^{(N)}(\tau)^2}, & a'_6 &:= \frac{\mathsf{E}_3(N\tau)}{864 \mathsf{E}_2^{(N)}(\tau)^3}
\end{align*}
in terms of the Hauptmodul $h_N$ when $X_0(N)$ has genus zero by comparing $q$-expansions, similar to how we found the rational functions giving $j$ in terms of each $h_N$. These give explicit expressions for the coefficients of the isogenous curves
\begin{align*}
	E: y^2 &= x^3 + a_4x + a_6\\
	E':y^2 &= x^3 + a_4'x + a_6'
\end{align*} in terms of the coordinate $h_N$. 
We give a complete set of results for all $$N \in\{ 2, 3, 4, 5, 6, 7, 8, 9, 10, 12, 13, 16, 18, 25\}$$ in Tables \ref{elliptic-coefficients-start} to \ref{eiliptic-coefficient-last} in the appendix. 

Finally, we include here some additional formulas for curves in special Weierstrass forms familiar from $2$- and $3$-isogeny descent. The curves 
\begin{align*}
 	E_2: y^2 = x^3 + a_2x^2 + a_4x, && E'_2 : y^2 = x^3 - 2a_2x^2 + (a_2^2 - 4a_4)x
 \end{align*} 
 are always $2$-isogenous, and the curves
 
 \begin{align*}
 	E_3: y^2 = x^3 + d(ax + b)^2 && E'_3: y^2 = x^3 -3d(ax + (27b -4a^3d)/9)^2
 \end{align*}
 are always $3$-isogenous in characteristic $0$; see [\cite{MR2514094}, Chapter X.1], and \cite{3iso} for more details on the role of these models in $2$- and $3$-isogeny descent, respectively. We have
 \begin{align*}
 	j(E_2) &= \frac{256(a_2^2 - 3a_4)^3}{(a_2^2a_4^2 - 4a_4^3)},\\
 	j(E_3) &= \frac{256(d^2a^4 - 6dab)^3}{4d^3a^3b^3 - 27d^2b^4}.
 \end{align*}
 Equating these $j$-invariants with the formula for the $j$-invariant in terms of Hauptmoduln from Table \ref{j-invar}
 \begin{align*}
 	j &= \frac{(h_2 + 256)^3}{h_2^2},\\
 	j &= \frac{(h_3 + 27)(h_3 + 243)^3}{h_3^3},
 \end{align*}
 we search for rational solutions for $h$. In both cases, we find only one rational solution given by
 \begin{align}\label{2-isogeny formula}
 	h_2 &= \frac{256 a_4}{a_2^2 - 4 a_4},\\
 	h_3 &= \frac{729b}{4a^3d - 27b}.\label{3-isogeny formula}
 \end{align}
 so in these models we explicitly find the points in $Y_0(2)$ and $Y_0(3)$ corresponding to the isogenies $E_2 \to E_2'$ and $E_3 \to E_3'$ respectively. Finally, we may invert expressions (\ref{2-isogeny formula}) and (\ref{3-isogeny formula}) to find that $E_2$ and $E_3$ correspond to the same points on $Y_0(2)$ and $Y_0(3)$ as
 \begin{align*}
 	\tilde{E}_2: y^2 &= x^3 + (4h_2 + 256)x^2 + h_2(4h_2 + 256)x,\\
 	\tilde{E}_3: y^2 &= x^3 + \frac{1}{4}((27h_3 + 729)x + h_3(27h_3 + 729)^2)^2,
 \end{align*}
 respectively, though $\tilde{E}_2$ might be a twist of $E_2$ and $\tilde{E}_3$ might be a twist of $E_3$.

\section{Elliptic $K$-curves}

\subsection{Introduction to $K$-curves}
Elliptic curves defined over $\Q$ have been extensively studied, and results such as Mazur's Theorem and the Modularity Theorem are some of the most celebrated in modern number theory. Therefore, one reasonable direction of study would be to find some generalization of elliptic curves over $\Q$ that would allow the extension of such known results. This has led to the study of elliptic $\Q$-curves, or more generally elliptic $K$-curves for any field $K$. 

\defn{Let $K$ be a field with separable closure $M$, and let $G = \Gal(M/K)$. An elliptic curve $E$ defined over $M$ is said to be a an \emph{elliptic $K$-curve}, or $K$-curve, if all $G$-conjugates of $E$ are isogenous to $E$ over $M$.
We let $L$ denote the field of definition of the $K$-curve $E$, so that $K \subseteq L \subseteq M$.}

\begin{example}\label{K-curve-ex}
	Let $E$ be the curve 
	\begin{align*}
	 	E: y^2 = x^3 - 2\sqrt{-2} \cdot x^2 + (-1 + 2\sqrt{-2})x.
	 \end{align*}
	 defined over $\Q(\sqrt{-2})$.
	 The curve $E$ is $2$-isogenous over $\Q(\sqrt{-2})$ to 
	 \begin{align*}
	 	E': y^2 = x^3 + 4\sqrt{-2} \cdot x^2 - (4 + 8\sqrt{-2})x
	 \end{align*}
	 Scale $x$ by $2$ and $y$ by $2\sqrt{2}$ to find that $E'$ is isomorphic over $\Q(i, \sqrt{2})$ to 
	 \begin{align*}
	 	\conj{E}: y^2 = x^3 + 2\sqrt{-2}\cdot x^2 - (1 + \sqrt{-2})x.
	 \end{align*}
	 Since $\conj{E}$ is the only conjugate of $E$ under $\Gal(\conj{\Q}/\Q)$ other than itself, we conclude that $E$ is a $\Q$-curve via the isogeny $E \to E' \to \conj{E}$. 
\end{example}

Notice that in this example the isomorphism between $E'$ and $\conj{E}$ had to be taken over an extension $\Q(i, \sqrt{2})$ of the field of definition $\Q(\sqrt{-2})$ of $E$. In the absence of complex multiplication, isogenies are unique up to composition with $[-1]$, so we cannot force this isogeny $E \to \conj{E}$ to be defined over the field of definition. Therefore, one natural question is:

\begin{question}\label{General-Q}
When can the isogeny $E \to E^\sigma$ between a $K$-curve $E$ and its conjugate be defined over the field of definition $L$ of $E$?
\end{question}
 We will examine this question shortly as one of the main results of this thesis. 

Galois representations coming from the Tate module $T_\ell E$ illustrate one reason why elliptic $\Q$-curves can be considered to generalize curves defined over $\Q$. We have a representation
\begin{align*}
	\rho_{C, \ell}: \Gal(\conj{\Q}/\Q) \to \End(T_\ell C)
\end{align*}
for any elliptic curve $C$ defined over $\Q$, and the isomorphism class of this representation depends only on the isogeny class of $C$. Similarly, letting $E$ be a $K$-curve, we may derive a Galois representation
\begin{align*}
	\rho_{E, \ell} :\Gal(\conj{\Q}/\Q)\to \conj{\Q}_\ell^*\GL_2(\Q_\ell).
\end{align*}

We omit further motivational discussion, since it would involve introducing Galois cohomology and would take us too far afield; see \cite{ellenberg} for further details. 
 
\subsection{Constructing and classifying $K$-curves}
Elliptic curves with complex multiplication provide the first examples of $K$-curves. Indeed, some authors, such as \cite{gross}, originally restricted the definition only to CM curves. We prove that CM curves are $K$-curves in the case that $K$ is a number field.
\begin{proposition}
	Let $K$ be a number field, and let $E/\conj{\Q}$ have complex multiplication. Then $E$ is a $K$-curve. 
\end{proposition}
\begin{proof}
	Since $E$ is CM, the endomorphism ring $\End_{\C}(E)$ must be an imaginary quadratic order $A$ [\cite{MR2514094}, Corollary III.9.4]. For any $\sigma \in \Gal(\conj{\Q}/K)$, the rings $\End_{\C}(E)$ and $\End_{\C}(E^\sigma)$ are isomorphic, since an endomorphism $\varphi$ of $E$ corresponds to the endomorphism $\varphi^\sigma$ of $E^\sigma$.  

	Now treat $E$ and $E^\sigma$ as complex elliptic curves $\C/\Lambda$ and $\C/\Lambda'$. Since endomorphisms of $E$ are in correspondence with homotheties of $\Lambda$ [\cite{MR2514094}, Corollary 4.1.1], we conclude that $\End_\C(\Lambda) \cong \End_{\C}(\Lambda') \cong A$. But the only complex lattices equipped with an endomorphism ring isomorphic to the imaginary quadratic order $A$ are those of the form $cI$ for some ideal $I$ of $A$ and $c \in \C$. Here, we are embedding $A \subset \C$ as a lattice itself. 

	We may therefore write $\Lambda = cI$ and $\Lambda = c'I$ for some ideals $I, I' \subseteq A \subset \C$ and $c, c' \in \C$. The ideal $I'$ has finite index $N'$, so it contains $N'A$. Therefore, $\Lambda' = c'I'$ contains the scalar multiple $\frac{c'N'}{c} I$ of $\Lambda$ with finite index, yielding a map $\C/\Lambda \to \C/\Lambda'$ of finite degree, which corresponds to an isogeny $E \to E'$.
\end{proof}

The set of $K$-curves is closed under isogeny, so that we may speak of isogeny classes of $K$-curves rather than isomorphism classes. 
\begin{proposition}
	If $E$ is a $K$-curve, then so is any curve $E'$ isogenous to $E$ over $M$. 
\end{proposition}
\begin{proof}
	Let $\sigma \in \Gal(M/K)$, and let $\varphi: E \to E^\sigma$ and $\psi: E \to E'$ be isogenies. Then we obtain an isogeny $\psi^\sigma: E^\sigma \to (E')^\sigma$, and so we can construct an isogeny $\psi^\sigma \circ \varphi \circ \hat{\psi}: E' \to (E')^\sigma$:
	\begin{center}
		\begin{tikzcd}
			E  \arrow[r, "\varphi"] & E^\sigma \arrow[d, "\psi^\sigma"] \\
			E'  \arrow[u, "\hat{\psi}"]  \arrow[r] & (E')^\sigma
		\end{tikzcd}
	\end{center}
\end{proof}

We may obtain $K$-curves from $K$-rational points on a modular curve $Y^*(N)$---recall that this modular curve is the the quotient of $Y_0(N)$ by the action of the group of Atkin-Lehner involutions $W(N)$. We may also obtain them similarly from $K$-rational points on $Y_0^+(N)$, the quotient of $Y_0(N)$ by the action of the Fricke involution.
\begin{proposition}{\label{k-construct}}
	Let $P$ be a $K$-rational point on $Y^*(N)$. Then every point in the fiber $\mathcal{P}$ of $P$ in $Y_0(N)$ corresponds to a $K$-curve. Similarly, the fibers of $K$-rational points on $Y^+_0(N)$ correspond to $K$-curves.
\end{proposition}
\begin{proof}
 	The set $\mathcal{P}$ is stable under $\Gal(M/K)$ because it is the fiber of a $K$-rational point under a $K$-rational map. However, since the points of $\mathcal{P}$ are all related by Atkin-Lehner involutions, their corresponding elliptic curves are all isogenous. 
	The same reasoning applies for the $K$-rational fibers of $Y_0^+(N)$.
 \end{proof} 

The previous proposition classifies all non-CM $K$-curves by a theorem of Elkies:
\begin{theorem}\label{elkies-theorem}
\emph{\cite{k-curves}.}  Let $K$ be any field with separable closure $M$, and let $E$ be a $K$-curve without complex multiplication. Then there exists a squarefree integer $N$, depending only on the $M$-isogeny class of $E$, such that:
\begin{enumerate}
 	\item $E$ is isogenous over $M$ with a $K$-curve arising from a $K$-rational point on $Y^*(N)$ in the sense of Proposition \ref{k-construct}; and
 	\item If for some $N'$ there is a $K$-rational point on $Y^*(N')$ that parameterizes $K$-curves isogenous with $E$, then $N \mid N'$. 
 \end{enumerate}

\end{theorem}
\begin{corollary}
	Any non-$CM$ $K$-curve $E$ is isogenous to a curve defined over a finite multi-quadratic extension $L/K$. That is, $L$ is of the form $K(\sqrt{a_1}, \sqrt{a_2}, \dots, \sqrt{a_n})$ with $a_i \in K$.
\end{corollary}
\begin{proof}
	This follows from the fact that the group $W(N)$ of Atkin-Lehner involutions is isomorphic to $(\Z/2\Z)^{\omega(N)}$. 
\end{proof}

\subsection{Main theorem on strict $K$-curves given by fibers over $Y_0^+(N)$}

We now revisit Question \ref{General-Q}. If $E$ is a $K$-curve defined over $L$, when may we define the isogeny $E \to E^\sigma$ over $L$? As discussed in Section \ref{K-rational}, the points on the modular curve $Y_0(N)$ correspond to $\C$-isomorphism classes of elliptic curves, but these split into many $K$-isomorphism classes related by quadratic twists as long as we avoid $j \in \{0, 1728\}$. We may ignore these two cases in our discussion because they correspond to CM curves. There is no canonical way to associate a $K$-isomorphism class of elliptic curves to a point $P \in Y_0(N)(K)$, yet this ambiguity also allows some freedom that we may leverage to our advantage. To illustrate, we reconsider Example 
\ref{K-curve-ex}.
\vspace{10pt}

\noindent\textbf{Example 4.1.2 revisited.}
	We found that the curve 
	$$E: y^2 = x^3 - 2\sqrt{-2}\cdot x^2 + (-1 + 2\sqrt{-2})x$$
	is isomorphic to its conjugate only over $\Q(i, \sqrt{2})$. However, let us instead consider the twist 
	\begin{align*}
		F: y^2 = x^3 + 4x^2 + (2 - 4\sqrt{-2})x,
	\end{align*}
	which is isomorphic to $E$ over $\C$ via the change of variables $(x, y) \mapsto \left(\sqrt{-2}\cdot x, \sqrt{-2\sqrt{-2}} \cdot y\right)$. Then $F$ is $2$-isogenous over $\Q(\sqrt{-2})$ to 
	\begin{align*}
		F': y^2 = x^3 - 8x^2 + 8(1 + 2\sqrt{-2})x.
	\end{align*}
	But $F'$ is isomorphic to 
	\begin{align*}
		\conj{F}: y^2 = x^3 + 4x^2 + (2 + 4\sqrt{-2})x
	\end{align*}
	via the change of variables $(x, y) \mapsto (-2x, -2\sqrt{-2} \cdot y)$. Thus, we obtain an isogeny $F \to \conj{F}$ that is defined \emph{over the field of definition} $L = \BQ(\sqrt{-2})$ of $F$, which was impossible for the twist $E$.

This example illustrates that by taking an appropriate twist of a $K$-curve, it may be possible to find a $K$-curve for which the isogeny between Galois conjugates is defined over the field of definition $L/K$ rather than requiring some extension of $L$.

\defn{Let $E$ be a $K$-curve defined over the extension $L/K$. We say that $E$ is a \emph{strict} $K$-curve if $E$ is isogenous to its $\Gal(L/K)$ conjugates over $L$ rather than some larger extension.}

Thus, we refine our main question to:
\begin{question}
	Let $K$ be a number field, and let $P$ be an $L$-rational point on $Y_0(N)$ corresponding to a $\C$-isomorphism class of $K$-curves. When can we \emph{choose} a twist giving a strict $K$-curve $E/L$ corresponding to $P$? When this is possible, can we determine exactly which twists of $E$ are strict $K$-curves?
\end{question}

 Let $h, h'$ denote $L$-rational coordinates on $X_0(N)$, even when $X_0(N)$ does not have genus $0$. We use bar notation, e.g. $\conj{h}$, to denote Galois conjugation over the quadratic extension $L/K$, which is not to be confused with complex conjugation. The fiber of any unramified point $H$ on $X^+(N)$ contains two curves $E, E'$; the ramified points correspond to CM curves, we which we ignore in our discussion. If $H$ is $K$-rational, then its two preimages $h, h'$ are either $K$-rational themselves and thus are automatically strict $K$-curves, else they are $L$-rational for some quadratic extension of $K$. In the latter case, the points $h$ and $h'$ must be conjugates in $L/K$. That is, we must have $$w^*h = \conj{h} = h',$$ where $\conj{h}$ denotes the unique Galois conjugate of $h$ in $L/K$.

With this setup, we present the main theorem of this thesis, which gives a precise constructive description of strict $K$-curves that arise from $K$-rational points on $X_0^+(N)$. This description involves a Diophantine condition on a pair of conics that is independent of the coordinate $h$, only depending on the field extension $L$. 

\begin{theorem}\label{main-theorem}
	 Let $L/K$ be a quadratic extension, and let $\{h, h'\} \subset Y_0(N)(L)\setminus Y_0(N)(K)$ be the fiber of a non-CM point in $Y_0^+(N)(K)$. Let $\tau, \tau' \in \mathbb{H}$ correspond to $h, h'$, and likewise let $E, E'$ be the isogenous elliptic curves corresponding to $h, h'$. 
	\begin{enumerate}[(a)]
		\item The following are equivalent:
		\begin{enumerate}[(i)]
			\item 
			There exists a choice of twist of cyclic $N$-isogeny $E \to E'$ defined over $L$ such that $E'$ is isomorphic over $L$ to the conjugate curve $\conj{E}$ of $E$.
		  \item At least one of the integers $N$ or $-N$ lies in the image of the Galois norm map $L^\times \to K^\times$.
		  \end{enumerate}
		\item The twists described by (a) are precisely those isomorphic to the models
		\begin{align*}
			E: y^2 = x^3 - \frac{\alpha^2 A_4}{48}x + \frac{\alpha^3A_6}{864}, && E': y^2 = x^3 - \frac{\alpha^2 A_4'}{48}x + \frac{\alpha^3 A_6'}{864},
		\end{align*}
		such that $\alpha \in L^\times$ has Galois norm $\alpha \conj{\alpha}$ lying in $-N\cdot (L^\times)^2$. Here, the coefficients $A_4, A_6, A'_4, A'_6$ are given by
		\begin{align*}
			A_4 &:= \frac{\mathsf{E}_4(\tau)}{\left(\mathsf{E}_2^{(N)}(\tau)\right)^2}, & A_6 &:= \frac{\mathsf{E}_6(\tau)}{\left(\mathsf{E}_2^{(N)}(\tau)\right)^3}, \\
			 A_4' &:= \frac{\mathsf{E}_4(N \tau)}{\left(\mathsf{E}_2^{(N)}(\tau)\right)^2},& A_6' &:= \frac{\mathsf{E}_6(N \tau)}{\left(\mathsf{E}_2^{(N)}(\tau)\right)^3}
		\end{align*}
		as in Section \ref{section3.3}. 
	\end{enumerate}	
\end{theorem}
\begin{proof}
	As discussed in Section \ref{section3.3}, the points $h, h'$ correspond to the pair of isogenous curves 
	\begin{align*}
		E: y^2 &= x^3 + a_4x + a_6 \\ E': y^2 &= x^3 + a_4'x + a_6'
	\end{align*}
	where
	\begin{align*}
		a_4 &= -\frac{E_4(\tau)}{48\lambda^2}, & a_6 &= \frac{\mathsf{E}_6(\tau)}{864\lambda^3},\\
		a_4' &= \frac{\mathsf{E}_4(N \tau)}{48\lambda^2}, & a_6' &= \frac{\mathsf{E}_6(N \tau)}{864\lambda^3}
	\end{align*}
	for an arbitrary modular form $\lambda$ of weight $2$ on $X_0(N)$; this pair is equipped with the cyclic $N$-isogeny $E \to E'$ defined over $L$ given by $h$. 

	We take the choice $\lambda = \alpha^{-1} \mathsf{E}_2^{(N)}(\tau)$ for some $\alpha \in L^{\times}$. As $\alpha$ ranges over $L^{\times}$, the models for $E$ and $E'$ range over all their quadratic twists. Then $\frac{\mathsf{E}_4(\tau)}{\mathsf{E}_2^{(N)}(\tau)^2}$ and $\frac{\mathsf{E}_6(\tau)}{\mathsf{E}_2^{(N)}(\tau)^3}$ are modular functions on $X_0(N)$. These functions are defined over $\Q$, and we have $h' = \conj{h} = w^*h$, so this implies that 
	\begin{align}\label{conjugate-iso}
		\conj{A_4} = \conj{\left(\frac{\mathsf{E}_4(\tau)}{\left(\mathsf{E}_2^{(N)}(\tau)\right)^2}\right)} &= {\frac{\mathsf{E}_4(-1/N\tau)}{\left(\mathsf{E}_2^{(N)}(-1/N\tau)\right)^2}},\\
		\conj{A_6}= \conj{\left(\frac{\mathsf{E}_6(\tau)}{\left(\mathsf{E}_2^{(N)}(\tau)\right)^3}\right)} &= {\frac{\mathsf{E}_6(-1/N\tau)}{\left(\mathsf{E}_2^{(N)}(-1/N\tau)\right)^3}},\label{conjugate-iso2}
	\end{align}
	since $\frac{\mathsf{E}_4(\tau)}{\mathsf{E}_2^{(N)}(\tau)^2}$ and $\frac{\mathsf{E}_6(\tau)}{\mathsf{E}_2^{(N)}(\tau)^3}$ are given by some rational expression in terms of the coordinate $h$. 

	The Eisenstein series $\mathsf{E_4}$ and $\mathsf{E}_6$ satisfy
	\begin{align*}
		\mathsf{E}_4(-1/N\tau) &= (N\tau)^4 \mathsf{E}_4(N\tau),\\
		\mathsf{E}_6(-1/N\tau) &= (N\tau)^6 \mathsf{E}_6(N\tau),
	\end{align*}
	and by Proposition \ref{E2-antiinvariant}, the modular form $\mathsf{E}_2^{(N)}$ is anti-invariant under the Fricke involution, so that
	\begin{align*}
		\mathsf{E}_2^{(N)}(-1/N\tau) = -N\tau^2\mathsf{E}_2^{(N)}(\tau).
	\end{align*}
	Combining these functional equations gives
	\begin{align*}
		{\frac{\mathsf{E}_4(-1/N\tau)}{\left(\mathsf{E}_2^{(N)}(-1/N\tau)\right)^2}} &= N^2 \frac{\mathsf{E}_4(N\tau)}{\left(\mathsf{E}_2^{(N)}(\tau)\right)^2} = N^2A_4',\\
		{\frac{\mathsf{E}_6(-1/N\tau)}{\left(\mathsf{E}_2^{(N)}(-1/N\tau)\right)^3}} &= -N^3 \frac{\mathsf{E}_6(N\tau)}{\left(\mathsf{E}_2^{(N)}(\tau)\right)^3} = -N^3 A_6',
	\end{align*}
	so that substituting equations (\ref{conjugate-iso}) and (\ref{conjugate-iso2}) yields 
	
	\begin{align}\label{fundamental-relation}
		\conj{A_4} &= (\alpha/\conj{\alpha})^2N^2 A_4', &
		\conj{A_6} &= -(\alpha/\conj{\alpha})^3 N^3 A_6'.
	\end{align}
	
	The only isomorphisms that preserve Weierstrass equations of the form $y^2 = x^3 + a_4x + a_6$ come from changes of variables of the form $$(x, y) \mapsto (u^2 x, u^3 y).$$ We need such a change of variables from 
	\begin{align*}
		E': y^2 = x^3 + a_4'x + a_6' 
	\end{align*}
	to 
	\begin{align*}
		\conj{E}: y^2 = x^3 + \conj{a}_4x + \conj{a}_6, 
	\end{align*}
	which necessitates 
	\begin{align*}
		u^4\conj{a_4} &= a_4',\\
		u^6\conj{a_6} &= a_6',
	\end{align*}
	or equivalently
	\begin{align*}
		u^4\conj{A_4} &= A_4',\\
		u^6\conj{A_6} &= A_6'.
		\end{align*}
The relations (\ref{fundamental-relation}) show that these conditions are equivalent to 
	\begin{align*}
		u^2 = - (\alpha/\conj{\alpha})N.
	\end{align*}
	We need $u^3 \in L^\times$ in order for the change of variables to occur over $L$. Therefore the isomorphism occurs over $L$ if and only if $- (\alpha/\conj{\alpha})N$ is a square in $L^\times$. Multiplying by the square $\conj{\alpha}^2/N^2$ gives the equivalent condition \begin{align}\alpha \conj{\alpha} \in -N \cdot (L^\times)^2, \label{main-condition}
	\end{align}
	yielding conclusion (b).

	To reach conclusion (a) from (b), we must determine when (\ref{main-condition}) has a solution over $L^\times$, which is equivalent to asking whether $-N$ is the product of a norm and a square. We state and prove the desired equivalent condition as a separate lemma:

	\begin{lemma}
		Let $L/K$ be a quadratic extension of characteristic not equal to $2$, and let $\mathbf{N}$ denote the image of the norm map $L^\times \to K^\times$.  Then an element $N \in K$ lies in $\mathbf{N} \cdot (L^\times)^2$ if and only if either $\pm N$ lies in $\mathbf{N}$. 
	\end{lemma}
	\begin{proof}
		We may write $L = K(\sqrt{D})$ for some $D \in K$, so $-1$ always lies in $\mathbf{N} \cdot (L^\times)^2$ as the product of the norm $-D$ of $\sqrt{D}$ and the square $D^{-1}$. Therefore, if either $N$ or $-N$ is a norm, then $N \in \mathbf{N} \cdot (L^\times)^2$.

		Conversely, suppose that $N$ lies in $\mathbf{N} \cdot (L^\times)^2$, so that we may write
		\begin{align*}
			N = \gamma^2 (a^2 - b^2D)
		\end{align*}
		for some $\gamma \in L^\times$ and $a, b \in K$. Over characteristic not equal to $2$, if $c, d\in K$ then an element of the form $(c + d\sqrt{D})^2$ lies in $K$ if and only $c = 0$ or $d = 0$. Since $N$ lies in $K$, we conclude that $\gamma$ must either lie in $K$ or in $\sqrt{D} \cdot K$.

		If $\gamma \in K$, then we may absorb the factor of $\gamma$ into $a$ and $b$ to conclude that $N$ lies in the image of the norm. Else $\gamma = d\sqrt{D}$ for some $d \in K$, so that
		\begin{align*}
			-N = (bD)^2 - (ad)^2D
		\end{align*}
		and thus $-N$ lies in the image of the norm map. 
	\end{proof}
	Since $-N$ is an integer and therefore lies in the base field $K$, conclusion (a) follows from (b) by the lemma.
\end{proof}

If we write $L = K(\sqrt{D})$, then part (a) states that the existence of a strict $K$ curve defined over $L$ corresponding to a $K$-rational point on $X_0^+(N)$ is equivalent to a solution to the pair of conics
\begin{align*}
	x^2 - Dy^2 = \pm N
\end{align*}
over $K$.

We conclude by illustrating the construction from part (b) of the theorem more explicitly in the case $N = 2$. Let $L = K(\sqrt{D})$. The values of $h = a + b\sqrt{D}$ satisfying the hypothesis of the theorem are those such that $\conj{h} = w_2^*h = 4096/h$, so that these values of $h$ are given precisely by the solutions to the conic
$$a^2 - Db^2 = 4096$$
with $a, b \in K$. We may parameterize this conic over $t \in K$ as
\begin{align*}
	a &= 64\cdot \frac{1 + Dt^2}{1 - Dt^2},\\
	b &= 128 \cdot \frac{t}{1 - Dt^2},
\end{align*}
and thus
\begin{align*}
	h = 64\cdot \frac{(t\sqrt{D} + 1)^2}{1 - Dt^2}.
\end{align*}
Substitute this parameterization of $h$ into the formulas in Table \ref{elliptic-coefficients-start} to find the coefficients of $E, E'$:
\begin{align*}
	E: y^2 &= x^3 - \alpha^2 \frac{(5 - 3t\sqrt{D})}{96}x + \alpha^3  \frac{(-7 + 9t\sqrt{D})}{1728},\\
	E': y^2 &= x^3 - \alpha^2 \frac{(5 + 3 t\sqrt{D})}{384}x + \alpha^3 \frac{(7 + 9t\sqrt{D})}{13824}.
\end{align*}
Therefore, as $t$ ranges over $K$ and $\alpha \in L$ ranges over all solutions to $\alpha \conj{\alpha} \in -2 \cdot (L^{\times})^2$, the two models above range over all twists of strict $K$-curves defined over $L$ corresponding to a $K$-rational fiber of $X_0^+(N)$. A similar construction may be given for other $N$ with $X_0(N)$ of genus $0$ as long as a rational parameterization of the conic
\begin{align*}
	a^2 - b^2D = \kappa_N
\end{align*}
is known, where $\kappa_N = h_N \cdot w_N^*h_N$. This is easier to do when $\kappa_N$ is a rational square, which only happens for $N = 2, 3, 4$, and $7$ (see Table \ref{hauptmoduln}). 

A natural way to extend Theorem \ref{main-theorem} would be to prove a version for fibers of $K$-rational points on $X^*(N)$ rather than $X_0^+(N)$, working over multi-quadratic extensions $L/K$ rather than quadratic extensions. Then the result would account for \emph{all} non-CM $K$-curves by Elkies' Theorem \ref{elkies-theorem}. One way to approach this problem would be to find suitable models for $E$ corresponding to $h \in Y_0(N)$ that transform nicely under the Atkin-Lehner involutions, since doing this for the Fricke involution was the key step in the proof of Theorem \ref{main-theorem}.

\appendix

\pagebreak

\section{Special values of the $j$-invariant at CM points}\label{j-appendix}

\begin{figure}[h]
	\centering\begin{tabular}{|c || c | c | c|}
		\hline $N$ & $-1/\tau = i\sqrt{N}$ & $h_N(\tau) = +\sqrt{\kappa_N}$ & $j(\tau) = j(-1/\tau)$ \\
		\hline
		 $2$ & $i\sqrt{2}$ & $64$ & $8000$\\
		 \hline
		 $3$ &  $i\sqrt{3}$ & $81$ & $54000$\\
		 \hline
		 $4$ & $2i$ & $16$ & $287496$ \\
		 \hline
		 $5$ & $i\sqrt{5}$ & $5\sqrt{5}$ & $632000 + 282880 \sqrt{5}$\\
		 \hline $6$ & $i\sqrt{6}$ & $6\sqrt{2}$ & $2417472 + 1707264 \sqrt{2}$\\
		 \hline $7$&$i\sqrt{7}$ & $7$ &  $16581375$\\
		 \hline $8$ &$2i\sqrt{2}$ & $4\sqrt{2}$ & $26125000 + 18473000 \sqrt{2}$\\
		 \hline $9$ &$3i$ & $3\sqrt{3}$ & $76771008 + 44330496 \sqrt{3}$\\
		 \hline $10$ &$i\sqrt{10}$ &$ 2\sqrt{5}$ & $212846400 + 95178240 \sqrt{5}$\\
		 \hline $12$ &$2i\sqrt{3}$ &$2\sqrt{3}$ & $1417905000 + 818626500 \sqrt{3}$\\
		 \hline $13$ &$i\sqrt{13}$ &$\sqrt{13}$ & $3448440000 + 956448000 \sqrt{13}$\\
		 \hline $16$ &$4i$ & $2\sqrt{2}$ & $ 41113158120 + 29071392966 \sqrt{2}$\\
		 \hline $18$ &$3i\sqrt{2}$ & $\sqrt{6}$ & $188837384000 + 77092288000 \sqrt{6}$ \\
		 \hline $25$ &$5i$ & $\sqrt{5}$ & $22015749611520 + 9845745509376 \sqrt{5}$\\
		 \hline
	\end{tabular}
	\caption{Special values of the $j$-invariant at CM points given by fixed points of the Fricke involution.}
	\label{special-j}
\end{figure}
\pagebreak
\section{Tables of coefficients of cyclic $N$-isogenies in terms of Hauptmoduln}\label{app-1}

These tables express the coefficients of an elliptic curve $E$ and its isogenous curve $E'$ corresponding to a point in $Y_0(N)$ given in terms of the Hauptmodul $h_N$ (see Table \ref{hauptmoduln}): 
\begin{align*}
	E: y^2 = x^3 - \frac{A_4}{48} + \frac{A_6}{864},\\
	E: y^2 = x^3 - \frac{A_4'}{48} + \frac{A_6'}{864},\\
\end{align*}
where 
\begin{align*}
	A_4 &:= \mathsf{E}_4(\tau)/\mathsf{E}_2(\tau)^2, & A_6 &:= \mathsf{E}_6(\tau)/\mathsf{E}_2(\tau)^3,\\
	A_4' &:= \mathsf{E}_4(N\tau)/\mathsf{E}_2(\tau)^2, & A_6' &:= \mathsf{E}_6(N\tau)/\mathsf{E}_2(\tau)^3.\\
\end{align*}
We omit the factors of $-1/48$ and $1/864$ in the tables. Additionally, the denominators of $A_4, A_6, A_4', A_6'$ often involve common factors that can be twisted away. When this applies, we label these factors as $D$ and instead give expressions for $D^2A_4, D^3A_6, D^2A_4', D^3A_6'$ in order to save space.

\begin{figure}[h]
\centering \scalebox{1.3}{\begin{tabular}{|c || c | c | c | c|} 
	\hline $\mathbf{N}$  & $A_4$ & $A_6$ & $A_4'$ & $A_6'$\\
	\hline $\mathbf{2}$ & $\frac{h + 256}{h + 64}$ & $\frac{h - 512}{h + 64}$ & $\frac{h + 16}{h + 64}$ & $\frac{h - 8}{h + 64}$\\
	
	\hline $\mathbf{3}$ & $\frac{h + 243}{h + 27}$ & $\frac{h^2 - 486h - 19683}{(h + 27)^2}$ & $ \frac{h + 3}{h + 27}$ & $ \frac{h^2 + 18h -27 }{(h + 27)^2}$\\

	\hline $\mathbf{4}$ & $\frac{h^2 + 256 h + 4096}{(h + 16)^2}$ & $\frac{(h + 32)(h^2 - 512h - 8192)}{(h + 16)^3}$ & $\frac{h^2 + 16h + 16}{(h + 16)^2}$ & $\frac{(h + 8)(h^2 + 16h - 8)}{(h + 16)^3}$\\

	\hline $\mathbf{5}$ & $\frac{h^2 + 250h + 3125}{h^2  +22 h + 125}$ & $\frac{h^2 - 500h - 15625}{h^2 + 22h + 125}$ & $\frac{h^2 + 10h + 5}{h^2 + 22h + 125}$ & $\frac{h^2 + 4h -1}{h^2 + 22h + 125}$\\
	\hline
	\end{tabular}
	}
	\caption{Coefficients of $E, E'$ corresponding to $h = h_N$ for $N = 2, 3, 4, 5$. }\label{elliptic-coefficients-start}
	\end{figure}

	\begin{figure}
	$\mathbf{N = 6}$

	\centering
	\scalebox{1.0}{\begin{tabular}{|c || c |}

	\hline $D^2A_4$ & $25\cdot {(h + 12)(h^3 + 252h^2 + 3888h + 15552)}$\\ 

	\hline $D^3A_6$ & $125\cdot {(h^2 + 36 h + 216)(h^4 - 504h^3 - 13824h^2 - 124416h - 373248)}$\\ 

	\hline $D^2A_4'$ & $25 \cdot {(h + 6)(h^3 + 18h^2 + 84h + 24)}$\\ 

	\hline  $D^3A_6'$ & $125 \cdot {(h^2 + 12h + 24)(h^4 + 24h^3 + 192 h^2 + 504h - 72)}$ \\

	\hline $D$ & $5h^2 + 84h + 360$ \\

	\hline
	\end{tabular}}
	\caption{Coefficients of $E, E'$ corresponding to $h = h_6$.}
	\end{figure}

\begin{figure}
	$\mathbf{N = 7}$

\centering
	\scalebox{1.3}{
	\begin{tabular}{|c||c|}
	\hline $A_4$ & $\frac{h^2 + 245h + 2401}{h^2 + 13h + 49}$ \\

	\hline $A_6$ & $\frac{h^4 - 490h^3 - 21609h^2 - 235298h -823543}{(h^2 + 13h + 49)^2}$ \\

	\hline $A_4'$ & $ \frac{h^2 + 5h + 1}{h^2 + 13h + 49}$ \\

	\hline $A_6'$ & $\frac{h^4 + 14h^3 + 63h^2 + 70h - 7}{(h^2 + 13h + 49)^2}$\\
	\hline
\end{tabular}
}
\caption{Coefficients of $E, E'$ corresponding to $h = h_7$.}
\end{figure}

\begin{figure}
$\mathbf{N = 8}$

\centering
\scalebox{1.2}{
\begin{tabular}{|c||c|}
	\hline $D^2A_4$ & $49 \cdot {(h^4 + 256h^3 + 5120h^2 + 32768h + 65536)}$\\

	\hline $D^3A_6$ & $343\cdot {(h^2 + 32h + 128)(h^4 - 512h^3 - 10240h^2 - 65536h - 131072)}$ \\

	 \hline $D^2A_4'$ & $49\cdot ({h^4 + 16h^3 + 80h^2 + 128h + 16})$\\

	\hline $D^3A_6'$ & $343 \cdot {(h^2 + 8h + 8)(h^4 + 16h^3 + 80h + 128h - 8)}$\\

	\hline $D$ & $7h^2 + 80h+ 224$\\

	\hline 
\end{tabular}
}
\caption{Coefficients of $E, E'$ corresponding to $h = h_8$.}
\end{figure}
\begin{figure}
$\mathbf{N = 9}$

\centering
\scalebox{1.1}{
\begin{tabular}{|c||c|}
	\hline $D^2A_4$ & ${(h + 1)(h^3 + 243h^2 + 2187h + 6561)}$\\

	\hline $D^3A_6$ &${h^6 - 486h^5 - 24057h^4 - 367416h^3 - 2657205h^2 - 9565938h - 14348907}$\\

	\hline $D^2A_4'$ & ${(h + 27)(h^3 + 9h^2 + 27h + 3)}$\\
	\hline $D^3A_6'$ & ${h^6 + 18h^5 + 135h^4 + 504h^3 + 891h^2 + 486h - 27}$\\
	\hline $D$ & $h^2 + 9h + 27$ \\
	\hline
\end{tabular}
}
\caption{Coefficients of $E, E'$ corresponding to $h = h_9$. }
\end{figure}

\begin{figure}
$\mathbf{N = 10}$

\centering \scalebox{1.2}{
\begin{tabular}{|c||c|}
	\hline  $D^2A_4$ & $9 \cdot \frac{h^6 + 260h^5 + 6400h^4 + 64000h^3 + 320000h^2 + 800000h + 800000}{h^2 + 8h + 20}$\\

	\hline$D^3A_6$ & $27\cdot \frac{(h^2 + 12h + 40)(h^2 + 30h + 100)(h^{4} - 520h^3 - 6600h^2 - 28000h - 40000)}{h^2 + 8h + 20}$ \\ 

	\hline$D^2A_4'$ & $9 \cdot \frac{h^6+20 h^5+160 h^4+640 h^3+1280 h^2+1040
   h+80}{h^2+8 h+20 }$ \\ 

  \hline $D^3A_6'$& $27\cdot \frac{\left(h^2+6 h+4\right) \left(h^2+6 h+10\right) \left(h^4+14
   h^3+66 h^2+104 h-4\right)}{h^2+8 h+20}$\\

   \hline $D$ & $3 h^2+26h+60$\\

   \hline
\end{tabular}
}
\caption{Coefficients of $E, E'$ corresponding to $h = h_{10}$.}
\end{figure}

\begin{figure}
$\mathbf{N = 12}$
\vspace{10pt}

\centering \scalebox{0.85}{
\begin{tabular}{|c||c|}
   \hline  

  $D^2A_4$ & $ 121 \cdot {(h^2 + 12h + 24)(h^6 + 252h^5 + 4392h^4 +31104h^3 + 108864h^2 + 186624h + 124416)}$ \\

     \hline $D^3A_6$ & \makecell[l]{$1331 \cdot (h^4 + 36h^3 + 288h^2 + 864h + 864)$ \\ \;\;\;\;\;\;\;\;\; $\cdot \; (h^8 - 504 h^7 - 14832 h^6 - 179712 h^5 - 1175040 h^4 - 4478976 h^3 - 
 9953280 h^2 - 11943936 h - 5971968)$} \\

  \hline $D^2A_4'$ & $121 \cdot {\left(h^2+6 h+6\right) \left(h^6+18 h^5+126 h^4+432 h^3+732
   h^2+504 h+24\right)}$ \\

   \hline $D^3A_6'$ & $1331 \cdot {\left(h^4+12 h^3+48 h^2+72 h+24\right) \left(h^8+24 h^7+240
   h^6+1296 h^5+4080 h^4+7488 h^3+7416 h^2+3024 h-72\right)}$\\
   \hline $D$ & $11
   h^4+156 h^3+816 h^2+1872 h+1584$\\

   \hline
\end{tabular}}

\caption{Coefficients of $E, E'$ corresponding to $h = h_{12}$.}
\end{figure}
\begin{figure}
$\mathbf{N = 13}$

\centering
\scalebox{1.2}{
\begin{tabular}{|c || c|}
	\hline  $A_4$ & $\frac{h^4 + 247h^3 + 3380h^2 + 15379h + 28561}{(h^2 + 5h + 13)(h^2 + 6h + 13)}$ \\

	 \hline $A_6$ & $\frac{h^6 - 494h^5 - 20618h^4 - 237276h^3 - 1313806h^2 - 3712930h - 4826809}{(h^2 + 5h + 13)^2(h^2 + 6h + 13)}$ \\

	 \hline $A_4'$ & $\frac{h^4 + 7h^3 + 20h^2 + 19h + 1}{(h^2 + 5h + 13)(h^2 + 6h + 13)}$ \\

	 \hline $A_6'$ & $\frac{h^6 + 10h^5 + 46h^4 + 108 h^3 + 122 h^2 + 38 h - 1}{(h^2 + 5h + 13)^2(h^2 + 6h + 13)}$ \\

	 \hline 
	 \end{tabular}
	 }
	 \caption{Coefficients of $E, E'$ corresponding to $h = h_{13}$.}
	 
	 \end{figure}

\begin{figure}
$\mathbf{N = 16}$
\vspace{10pt}

\centering \scalebox{0.9}{
\begin{tabular}{|c || c|}
	\hline  $D^2A_4$ & $25\cdot {(h^8 + 256 h^7 + 5632 h^6 + 53248 h^5 + 282624 h^4 + 917504 h^3 + 
 1835008 h^2 + 2097152 h + 1048576)}$ \\

\hline
  $D^3A_6$ & \makecell[l]{$125 \cdot (h^4 + 32 h^3 + 192 h^2 + 512 h + 512)$\\ \;\;\;\;\;\;\; $\cdot \; (h^8 - 512 h^7 - 11264 h^6 - 106496 h^5 - 565248 h^4 - 1835008 h^3 - 
 3670016 h^2 - 4194304 h - 2097152 )$} \\
 \hline

  $D^2A_4'$ & $25 \cdot { \left(h^8+16 h^7+112 h^6+448 h^5+1104 h^4+1664 h^3+1408
   h^2+512 h+16\right)}$ \\
   \hline

    $D^3A_6'$ & $125 \cdot {\left(h^4+8 h^3+24 h^2+32 h+8\right) \left(h^8+16 h^7+112
   h^6+448 h^5+1104 h^4+1664 h^3+1408 h^2+512 h-8\right)}$\\

   \hline $D$ & $\left(h^2+4
   h+8\right) \left(5 h^2+28 h+40\right)$
   \\ \hline
\end{tabular}}
\caption{Coefficients of $E, E'$ corresponding to $h = h_{16}$.}
   \end{figure}
\begin{figure}
$\mathbf{N = 18}$
\vspace{10pt}

\centering\scalebox{0.85}{

\begin{tabular}{|c || c|}
   \hline 
   
   \hline $D^2A_4$ & \makecell[l]{$289\cdot (h^3 + 12h^2 + 36h+36)$ \\ \;\;\;\;\;\;\;\; $\cdot \;(h^9 + 252 h^8 + 4644 h^7 + 39636 h^6 + 198288 h^5 + 629856 h^4 + 
 1294704 h^3 + 1679616 h^2 + 1259712 h + 419904)$ } \\
 
  \hline $D^3A_6$ & \makecell[l]{$4913 \cdot (h^6 + 36 h^5 + 324 h^4 + 1404 h^3 + 3240 h^2 + 3888 h + 1944 ) $ \\ \;\;\;\;\;\;\;\;$\cdot \; (h^{12} - 504 h^{11} - 15336 h^{10} - 208872 h^9 - 1700352 h^8 - 9206784 h^7 - 
 34836480 h^6$ \\ $\;\;\;\;\;\;\;\;\;\;\;\;\;\;\;\; - \;94058496 h^5 + 181398528 h^4 - 245223936 h^3 + 
 221709312 h^2 - 120932352 h - 30233088 )$} \\
 
  \hline $D^2A_4'$ & $289\cdot \left(h^3+6 h^2+12 h+6\right) \left(h^9+18 h^8+144 h^7+666
   h^6+1944 h^5+3672 h^4+4404 h^3+3096 h^2+1008 h+24\right)$ \\
   
   \hline $D^3A_6'$ & \makecell[l]{$ 4913 \cdot \left(h^6+12 h^5+60 h^4+156 h^3+216 h^2+144 h+24\right)$ \\ $\;\;\;\;\;\;\;\;
   \cdot \; (h^{12}+24 h^{11}+264 h^{10}+1752 h^9+7776 h^8+24192 h^7+53760
   h^6 $ \\ \;\;\;\;\;\;\;\;\;\;\;\;\;\;\;\; $+\;85248 h^5+94464 h^4+69624 h^3+30672 h^2+6048
   h-72)$ }\\
   \hline $D$ & $ 17 h^6+228 h^5+1332 h^4+4284 h^3+7992 h^2+8208
   h+3672$ \\
   \hline 
   \end{tabular}}
   \caption{Coefficients of $E, E'$ corresponding to $h = h_{18}$.}
   \end{figure}
   \begin{figure}
$\mathbf{N = 25}$
\vspace{10pt}

   \centering \scalebox{0.88}{
\begin{tabular}{|c ||c|}

	\hline $D^2A_4$ & \makecell[l]{$(h^{10} + 250h^9 + 4375h^8 + 35000h^7 + 178125h^6 + 631250h^5 + 1640625h^4$ \\ \;\;\;\;\;\;\; $ + 3125000h^3 + 4296875h^2 + 3906250 h + 1953125)/(h^2 + 2h + 5)$} \\

	\hline $D^3A_6$ & \makecell[l]{$(h^4 + 10h^3 +45h^2 + 100h + 125)$ \\ $  \;\;\;\;\;\;\; \cdot \;(h^{10} - 500h^9 - 18125h^8 - 163750h^7 - 871875h^6 - 3137500h^5 - 8203125h^4$ \\  \;\;\;\;\;\;\; \;\;\;\;\;\;\;$-\; 15625000h^3 - 21484375h^2 - 19531250h-9765625)/(h^2 + 2h + 5)$ }\\

	\hline $D^2A_4'$ & $(h^{10}+10 h^9+55 h^8+200 h^7+525 h^6+1010 h^5+1425 h^4+1400 h^3+875
   h^2+250 h+5)/(h^2 + 2h + 5)$ \\

   \hline $D^3A_6'$ & \makecell[l]{$\left(h^4_{25}+4 h^3+9 h^2+10 h+5\right)$ \\  \;\;\;\;\;\;\; $\cdot \left(h^{10}+10 h^9+55
   h^8+200 h^7+525 h^6+1004 h^5+1395 h^4+1310 h^3+725 h^2+100
   h-1\right)/(h^2+2 h+5)$}\\

   \hline $D$ & $h^4 + 5h^3 + 15h^{2} + 25h + 25$\\

   \hline
\end{tabular}

  }
  \caption{Coefficients of $E, E'$ corresponding to $h = h_{25}$.}
  \label{eiliptic-coefficient-last}
\end{figure}

\pagebreak
{\raggedright \printbibliography}

\end{document}